\documentclass[pdflatex,sn-mathphys-num]{sn-jnl}

\usepackage{amsmath,amssymb,amsthm}
\usepackage{geometry}
\newtheorem{theorem}{Theorem}
\newtheorem{lemma}{Lemma}
\newtheorem{proposition}{Proposition}
\newtheorem{corollary}{Corollary}
\newtheorem{remark}{Remark}

\usepackage{booktabs}
\usepackage{graphicx} 
\usepackage{float}      
\usepackage{placeins}   
\usepackage{xcolor}
\usepackage{url}
\usepackage{subcaption}
\usepackage{array}
\usepackage{tabularx}
\usepackage{comment}
\usepackage{algorithm}
\usepackage{algpseudocode}
\newcommand{\R}{\mathbb{R}}

\usepackage{caption}
\usepackage[english]{babel}

\begin{document}

\title[Scalable $s$--step PCG]{Scalable $s$--step Preconditioned Conjugate Gradient with Chebyshev Basis and Gauss--Seidel Gram Solve}


\author*[1]{\fnm{Pasqua} \sur{D'Ambra}}\email{pasqua.dambra@cnr.it}

\author[2]{\fnm{Massimo} \sur{Bernaschi}}\email{massimo.bernaschi@cnr.it}

\author[2]{\fnm{Mauro Giovanni} \sur{Carrozzo}}\email{maurogiovanni.carrozzo@cnr.it}

\author[3]{\fnm{Stephen} \sur{Thomas}}\email{sjt223@lehigh.edu}

\affil*[1]{\orgdiv{Institute for Applied Computing ``Mauro Picone''}, \orgname{National Research Council of Italy}, \orgaddress{\street{Via Pietro Castellino, 111}, \city{Naples}, \postcode{80131}, \country{Italy}}}

\affil[2]{\orgdiv{Institute for Applied Computing ``Mauro Picone''}, \orgname{National Research Council of Italy}, \orgaddress{\street{Via dei Taurini, 19}, \city{Rome}, \postcode{00185}, \country{Italy}}}

\affil[3]{\orgdiv{Department of Computer Science and Engineering}, \orgname{Lehigh University}, \orgaddress{\street{27 Memorial Drive West}, \city{Bethlehem}, \postcode{18015}, \state{PA}, \country{USA}}}


\abstract{We present a variant of the $s$--step Preconditioned Conjugate
Gradient (PCG) method that combines a Chebyshev--stabilized Krylov
basis with a Forward Gauss--Seidel (FGS) iteration for the solution of
the reduced Gram systems. In $s$--step PCG, multiple search directions
are generated per outer iteration, reducing global synchronization
costs but requiring the solution of small dense Gram systems whose
conditioning is critical for stability.

For a raw Chebyshev Krylov block, we derive an exact
moment--based representation of the associated Gram matrix in the
unpreconditioned setting and discuss its extension to left
preconditioning. This representation provides a structural
interpretation of the favorable Gram--matrix behavior observed for
moderate step sizes. Building on inexact Krylov theory and on the
algebraic connection between FGS and Modified Gram--Schmidt (MGS), we
derive residual bounds for the reduced solves and provide a rationale
for using a fixed, moderate number of FGS sweeps. Numerical
experiments assess the resulting accuracy of the Gram solves and the
convergence behavior of the outer iteration.

Large--scale experiments on modern NVIDIA GPU architectures,
including weak-scaling tests with aggregation-based
Algebraic MultiGrid (AMG) preconditioning, show that, within the
tested configurations, the proposed Chebyshev--stabilized $s$--step
PCG method with FGS Gram solves achieves convergence comparable to
classical PCG while reducing global synchronization overhead. The
results demonstrate its potential as a scalable
communication--reducing solver for current and next--generation accelerator systems.}



\maketitle

\section{Introduction}
\label{sec:intro}

We consider the numerical solution of large, sparse, symmetric
positive--definite (SPD) linear systems
\[
    Ax = b,
    \qquad A \in \mathbb{R}^{n\times n},
    \quad x, b \in \mathbb{R}^n.
\]
Such systems arise ubiquitously in scientific and engineering applications
including discretized partial differential equations, optimization, and
data analysis.
For large--scale problems, direct solvers are often impractical due to
their memory footprint and superlinear complexity, making iterative
methods essential.
Among them, the Preconditioned Conjugate Gradient (PCG) method is the
standard algorithm for SPD systems.
Let $M$ be an SPD preconditioner and define the (left) preconditioned operator
\[
\mathcal{A} := M^{-1}A.
\]
Given an initial guess $x_0$ and residual $r_0 = b - Ax_0$, PCG constructs approximations in the affine spaces
\[
x_0+\mathcal{K}_m(\mathcal{A},M^{-1}r_0),
\]
where
\[
\mathcal{K}_m(\mathcal{A},M^{-1}r_0)
=
\operatorname{span}
\left\{
M^{-1}r_0,\,
\mathcal{A}M^{-1}r_0,\,
\dots,\,
\mathcal{A}^{m-1}M^{-1}r_0
\right\}.
\]
The convergence of PCG depends directly on the spectrum of $\mathcal{A}$ and thus on the quality of the preconditioner.
The choice of the preconditioner is therefore a critical design aspect, as it strongly affects both the convergence rate and the overall computational cost of the method.
In this work, however, we do not address the design of the preconditioner itself; instead, we focus on the intrinsic limitations of the PCG solver, which persist even in the presence of highly optimized preconditioners and stem from its algorithmic structure.

In particular, PCG relies on two additional performance--critical computational kernels: the sparse matrix--vector product (SpMV) and the computation of inner product (dot).
SpMV is inherently memory--bandwidth bound and requires communication to exchange boundary data among neighboring processes, with a communication cost that depends on the sparsity pattern of the matrix and the underlying data partitioning.
Importantly, this communication is typically local and structured, and several SpMV implementations have been shown to effectively overlap communication with computation, thereby mitigating its impact on scalability.
In contrast, the computation of dot operations entails global reduction operations, which impose global synchronization and introduce communication
latencies that are difficult to hide completely, thereby often
emerging as a primary scalability bottleneck on massively parallel
systems.
These considerations motivate the development of Communication--Avoiding (CA) Krylov subspace methods, and in particular $s$--step formulations, which trade additional local computation for a significant reduction in the number of global synchronizations, thereby improving strong scalability on modern HPC architectures.

In $s$--step PCG~\cite{ChronGear1989p,ChronGear1989}, the information from
$s$ consecutive Krylov iterations is aggregated into a single block.
For example, setting $u_0=M^{-1}r_0$, a raw monomial Krylov block has
the form
\[
U=
\left[
u_0,\,
\mathcal{A}u_0,\,
\dots,\,
\mathcal{A}^{s-1}u_0
\right],
\]
and the corresponding block of $A$--conjugate directions is obtained from
a reduced Gram system involving $U$ and the previously computed search
directions.

A well--known difficulty of this monomial construction is numerical
instability: the vectors $\mathcal{A}^j u_0$ rapidly lose linear
independence as $j$ increases, causing the Gram matrix in the $A$--inner
product, $W = U^\top A U$, to become severely ill--conditioned.
To mitigate this effect, the
monomial basis can be replaced by a shifted and scaled Chebyshev
basis,
\[
u_j=T_j(\widehat{\mathcal{A}})u_0,
\qquad
j=0,\dots,s-1,
\]
where $T_j$ denotes the Chebyshev polynomial of the first kind and
$\widehat{\mathcal{A}}$ is a shifted and scaled form of
$\mathcal{A}$, defined more precisely in
Section~\ref{sec:sPCG-cheb}.
Under assumptions including normality of the matrix, a particular
spectral distribution of the starting vector, and an appropriate
spectral region, Philippe and Reichel~\cite{Reichel2012} derive
favorable conditioning bounds for raw Chebyshev Krylov bases. Their
results provide the theoretical motivation for adopting a Chebyshev
basis in place of the monomial basis in the present work.

Two reduced Gram systems must be solved at every outer iteration.
To solve them efficiently, we use a fixed, moderate number of Forward
Gauss--Seidel (FGS) sweeps. This choice is motivated by the algebraic connection between one FGS
sweep on a Gram system and a sequential Modified Gram--Schmidt (MGS)
projection process for a particular right-hand side. Although this
equivalence does not apply literally to the general right-hand sides
arising in PCG--S, it provides a structural interpretation of the
triangular Gram splitting used by FGS.
An earlier version of the FGS--based Gram--solver analysis appeared in~\cite{ThomasDAmbra2025}. The present work extends it with a structural moment-based analysis of
Chebyshev Gram matrices, a performance model for multi--GPU architectures,
and a large-scale experimental evaluation within the BootCMatchGX
software framework.

Throughout this work, the preconditioner is regarded
as fixed. For a given preconditioned operator, the choice of the
polynomial basis determines the coordinate representation of the
local Krylov blocks and can significantly affect the conditioning of
the associated reduced Gram matrices. Our analysis focuses on this
basis-induced algebraic structure, which is distinct from the
application-dependent problem of constructing the preconditioner
$M$.

The main contributions of this work can be summarized as follows.

\begin{itemize}
\item We propose a scalable $s$--step PCG formulation that combines
a Chebyshev--stabilized Krylov basis with FGS sweeps for the reduced
Gram solves. The formulation targets multi--GPU systems by reducing
global synchronization points and reformulating low
arithmetic--intensity BLAS--1 operations as block BLAS--2/3 kernels
(matrix--vector and matrix--matrix products), enabling efficient
utilization of GPU throughput.

\item For the unpreconditioned case, we derive an exact moment-based
representation of the Gram matrix associated with a raw Chebyshev
Krylov block. This representation identifies how the polynomial basis
and the induced spectral distribution affect the matrix entries and
provides a structural interpretation of the favorable conditioning
observed for moderate step sizes. We also discuss the corresponding interpretation in
terms of the symmetrically preconditioned operator.

\item Building on inexact Krylov theory and on the algebraic
connection between FGS and MGS, we formulate a practical
reduced-residual criterion for controlling the accuracy of the
inexact Gram solves. Numerical experiments assess its effectiveness
in preserving the observed convergence behavior of the outer
$s$--step PCG iteration.

\item We develop a performance model that quantifies the trade-off
between reduced global synchronization and increased local
computation. We compare its predictions with strong- and weak-scaling
experiments on modern multi--GPU supercomputers.

\item To the best of our knowledge, this is the first publicly available fully
distributed multi--GPU implementation and large-scale experimental
evaluation of preconditioned $s$--step CG. The method is implemented
within the open-source BootCMatchGX framework\footnote{\url{https://github.com/bootcmatch/BootCMatchGX}}, enabling reproducibility
and further research on scalable communication--reducing Krylov methods.
\end{itemize}

The remainder of the paper is organized as follows.
Section~\ref{sec:sPCG-cheb} introduces the $s$--step PCG method with Chebyshev Krylov basis.
Section~\ref{sec:erroranalysis} analyzes the inexact solution of the Chebyshev Gram systems and discusses its stability and convergence implications.
Section~\ref{sec:impl} summarizes details on our practical implementation of the method and presents the performance model.
Section~\ref{sec:related} reviews related work.
Section~\ref{sec:results} reports numerical results on large--scale GPU platforms, and Section~\ref{sec:conc} concludes the paper with final remarks and future work.

\section{The $s$--step PCG method with Chebyshev Krylov basis}
\label{sec:sPCG-cheb}

The $s$--step preconditioned Conjugate Gradient (PCG--S) method considered in this work is summarized in Algorithm~\ref{alg:pcgs}.
Given an initial guess $x^{(0)}$ and residual $r^{(0)} = b - Ax^{(0)}$, the PCG--S method proceeds by grouping
$s$ consecutive Krylov iterations into a single outer iteration.
At each outer iteration $k$, a block Krylov subspace of dimension $s$ is
constructed using a Chebyshev polynomial basis associated with the current
residual $r^{(k)}$.

Let $\lambda_{\min}$ and $\lambda_{\max}$ denote the
values used to scale the preconditioned operator
$\mathcal{A}=M^{-1}A$, and define
\[
\widehat{\mathcal A}
=
\frac{2\mathcal{A}-(\lambda_{\max}+\lambda_{\min})I}
     {\lambda_{\max}-\lambda_{\min}}.
\]
Since $A$ and $M$ are SPD, the eigenvalues of $\mathcal{A}$ are real
and strictly positive. In our implementation, we set
$\lambda_{\min}=0$ and estimate $\lambda_{\max}$ once, before
starting the PCG--S iteration, by applying the power method to
$\mathcal{A}$. The power iteration is terminated when
$
\left|\lambda_{\max}^{(\ell)}
-\lambda_{\max}^{(\ell-1)}\right|
<
10^{-3}\left|\lambda_{\max}^{(\ell)}\right|,$
or after $50$ iterations. Each step requires one sparse
SPMV with $A$ and one application of $M^{-1}$.
This inexpensive initialization provided sufficiently robust scaling
in all the numerical experiments reported in this work.

The Chebyshev Krylov basis is defined as
\[
Z_k =
\bigl[
z_1^{(k)}, \dots, z_s^{(k)}
\bigr],
\qquad
z_j^{(k)}
=
T_{j-1}(\widehat{\mathcal A})M^{-1}r^{(k)},
\]
where $T_j$ denotes the Chebyshev polynomial of the first kind of
degree $j$. In particular, $z_1^{(k)}=M^{-1}r^{(k)}$.
The construction of $Z_k$, together with the augmented product matrix
\[
  P_k = \bigl[r^{(k)},\; Az_1^{(k)},\;\dots,\; Az_s^{(k)}\bigr] \in \mathbb{R}^{n\times(s+1)},
\]
whose first column carries the current residual whereas the remaining columns contain the products $AZ_k$, is obtained by means of the so--called Matrix Power Kernel (MPK), as described in Algorithm~\ref{alg:MPK}.
The parameters $\alpha$, $\sigma$, and $\gamma$, appearing in
Algorithm~\ref{alg:MPK} and collected in the \texttt{basis\_info} parameter, follow from the three--term recurrence
satisfied by the Chebyshev polynomials,
\[
T_{j+1}(x)=2xT_j(x)-T_{j-1}(x).
\]
Since $\widehat{\mathcal A}=\alpha M^{-1}A-\sigma I,$
the recurrence for the basis vectors is
\[
z_{j+1}
=
2\alpha M^{-1}Az_j
-
2\sigma z_j
-
\gamma z_{j-1},
\]
where
\[
\alpha =
\frac{2}{\lambda_{\max}-\lambda_{\min}},
\qquad
\sigma =
\frac{\lambda_{\max}+\lambda_{\min}}
     {\lambda_{\max}-\lambda_{\min}},
\qquad
\gamma=1.
\]

The block of search directions is denoted by
$Q_k\in\mathbb{R}^{n\times s}$. 
The approximate solution is updated as
\[
x^{(k+1)}=x^{(k)}+Q_k\alpha_k,
\]
where $\alpha_k\in\mathbb{R}^s$ is computed from the
reduced Gram system
\[
W_k\alpha_k=m_k,
\qquad
W_k:=Q_k^\top A Q_k,
\qquad
m_k:=Q_k^\top r^{(k)}.
\]
In the proposed method, the
system is instead solved approximately by a prescribed number of FGS sweeps.

The residual is then updated according to
\[
r^{(k+1)} = r^{(k)} - A Q_k \alpha_k .
\]

To construct a new block that is $A$--conjugate to the previous
search-direction block, the raw Chebyshev block $Z_{k+1}$ is corrected
by solving a second reduced Gram system with multiple right-hand
sides. Specifically, let
\begin{equation}
\label{blockdirupdate}
Q_{k+1}=Z_{k+1}+Q_k\beta_k,
\end{equation}
where $\beta_k\in\mathbb{R}^{s\times s}$. Imposing the conjugacy condition $Q_{k+1}^\top A Q_k=0$ and using the symmetry of $W_k=Q_k^\top A Q_k$
gives
\[
W_k\beta_k=-Q_k^\top A Z_{k+1}.
\]
In Algorithm~\ref{alg:pcgs}, the product $AZ_{k+1}$ is stored in the
variable $AQ_{\mathrm{new}}$. Hence, the right-hand side is assembled
at line~13 as $
b_2\gets -Q^\top AQ_{\mathrm{new}}.$

The new search-direction block is then obtained from
Equation~\eqref{blockdirupdate}, while its product with $A$ is updated
as
\[
A Q_{k+1}
=
A Z_{k+1}
+
(AQ_k)\beta_k.
\]
These two block updates correspond to line~15 of
Algorithm~\ref{alg:pcgs}.
When the reduced Gram system for $\beta_k$ is solved exactly, this
update is the direct block analogue of the classical CG recurrence
and enforces $A$--conjugacy between consecutive search-direction
blocks, while avoiding additional applications of $A$. With the
inexact FGS solution employed in this work, this conjugacy condition
is satisfied only approximately, with an accuracy governed by the
residual of the reduced Gram solve.

This block formulation also has important computational implications.
In classical CG, most algebraic operations consist of Level 1 Basic Linear Algebra Subprograms (BLAS--1) vector kernels
(dot products and axpy updates), which have low arithmetic intensity and
limited efficiency on modern accelerators.
In contrast, the $s$--step formulation naturally aggregates vectors into blocks,
so that these operations can be reformulated as matrix--vector and matrix--matrix
products involving the block bases $Z_k$, $Q_k$, and $AQ_k$.
As a result, many BLAS--1 operations are replaced by BLAS--2/3 kernels
(e.g., GEMV and GEMM), increasing arithmetic intensity and enabling
more efficient utilization of GPU architectures.

\begin{algorithm}[H]
\caption{Preconditioned $s$--step Conjugate Gradient (PCG--S) Algorithm}
\label{alg:pcgs}
\begin{algorithmic}[1]
\Require $A \in \R^{n \times n}$, $b \in \R^n$, $x^{(0)}$, preconditioner $M$, step size $s$, tolerance $\tau$, maximum iterations $k_{\max}$, number of FGS sweeps $\nu$
\Ensure Approximate solution $x^{(k)}$
\State $r^{(0)} \gets b - A x^{(0)}$
\State $[Z,P] \gets \text{MPK}(A,r^{(0)},s,M,\text{basis\_info})$
\State $Q \gets Z$, \quad $AQ \gets P(:,2:\text{end})$
\For{$k=0,1,\dots,k_{\max}-1$}
   \State $W \gets Q^\top AQ$, \quad $b_1 \gets Q^\top P(:,1)$
   \State $\alpha \gets \mathrm{FGS}(W,b_1,\nu)$
   \State $x^{(k+1)} \gets x^{(k)} + Q \alpha$
   \State $r^{(k+1)} \gets r^{(k)} - AQ \alpha$
   \If{$\|r^{(k+1)}\| \leq \tau \|b\|$} \textbf{stop}
   \EndIf
   \State $[Z,P] \gets \text{MPK}(A,r^{(k+1)},s,M,\text{basis\_info})$
   \State $AQ_{\text{new}} \gets P(:,2:\text{end})$
   \State $b_2 \gets - Q^\top AQ_{\text{new}}$
   \State $\beta \gets \mathrm{FGS}(W,b_2,\nu)$
   \State $Q \gets Z + Q \beta$, \quad $AQ \gets AQ_{\text{new}} + AQ \beta$
\EndFor
\end{algorithmic}
\end{algorithm}
\begin{algorithm}[H]
\caption{\textsc{MPK} for Chebyshev basis}
\label{alg:MPK}
\begin{algorithmic}[1]
\Require $A\in\R^{n\times n}$, $r\in\R^n$, step size $s\ge 1$, preconditioner $M$, \text{basis\_info}
\Ensure Matrices $Z=[z_1,\dots,z_s]$, $P=[p_1,\dots,p_s,p_{s+1}]$
\State $p_1 \gets r$; \quad $z\!p_1 \gets p_1$; \quad $z_1 \gets  M^{-1} z\!p_1$
\If{$s>1$}
  \State $p_2 \gets A z_1$
  \State $z\!p_2 \gets \alpha\,p_2 - \sigma\,z\!p_1$
  \State $z_2 \gets  M^{-1} z\!p_2$
  \If{$s>2$}
    \For{$q=2$ \textbf{to} $s-1$}
      \State $p_{q+1} \gets A z_q$
      \State $z\!p_{q+1} \gets 2\alpha\,p_{q+1} - 2\sigma\,z\!p_q - \gamma\,z\!p_{q-1}$
      \State $z_{q+1} \gets M^{-1} z\!p_{q+1}$
    \EndFor
  \EndIf
\EndIf
\State $p_{s+1} \gets A z_s$
\State \textbf{return} $Z=[z_1,\dots,z_s]$, \; $P=[p_1,\dots,p_s,p_{s+1}]$
\end{algorithmic}
\end{algorithm}

In the following, we analyze the use of a fixed moderate number of FGS
sweeps (function FGS at lines~6 and~14 of
Algorithm~\ref{alg:pcgs}) to approximate the reduced Gram solves
arising in PCG--S and identify the conditions under which the
resulting inexact inner solver remains compatible with the
convergence of the outer iteration.

\section{Iterated Gauss--Seidel solution of the Chebyshev Gram systems}
\label{sec:erroranalysis}

General inexact Krylov theory
\cite{vandenEshofSleijpen2004,bouras2005}
shows that outer convergence can be preserved when the perturbations
introduced during the iteration are controlled relative to the
current residual. Following this general principle, we monitor the
accuracy of the approximate Gram solve through the practical
relative-residual criterion
\begin{equation}
\|m_k-W_k\alpha_k\|_2
\leq
\delta_k\|m_k\|_2,
\label{cond:inexactKrylov1}
\end{equation}
where $\delta_k$ controls the accuracy of the reduced solve. The
relation between this criterion and the rigorous perturbation bounds
of inexact Krylov theory is discussed in
Remark~\ref{remark:inexactkrylov}.

Under the uniform bound $\delta_k\leq\delta$, we use the heuristic
guideline
\begin{equation}
N_{\mathrm{outer}}\,\delta\lesssim1,
\label{cond:inexactKrylov2}
\end{equation}
to prevent the accumulated inner inaccuracies from becoming
dominant over the course of the outer iteration. This guideline is
not claimed to be a rigorous sufficient condition.

\begin{remark}
\label{remark:inexactkrylov}
We emphasize that condition~\eqref{cond:inexactKrylov1} represents a practical tolerance for the inner Gram solve, while condition~\eqref{cond:inexactKrylov2} should be interpreted as a heuristic
guideline rather than as a rigorously derived bound. A fully rigorous inexact Krylov analysis
(see~\cite{vandenEshofSleijpen2004,bouras2005})
would control the perturbation relative to $\|r^{(k)}\|$ rather than
$\|m_k\|$. Since $m_k=Q_k^\top r^{(k)}$, relating these quantities requires
additional control of the relation between
$\|Q_k^\top r^{(k)}\|$ and $\|r^{(k)}\|$. Such control is not
established here, and we therefore do not claim that the practical
criterion is equivalent to the perturbation condition used in the
general theory of inexact Krylov methods. Instead, the condition~\eqref{cond:inexactKrylov2} is used as a
practical guideline for selecting the inner accuracy. Its suitability in
the present context is supported by the numerical experiments reported in
Sections~\ref{sec:gram-experiments} and~\ref{sec:results}.
\end{remark}

\begin{remark}
Throughout this section, the outer iteration index~$k$ is omitted for notational simplicity. The analysis is carried out for the Gram system defining $\alpha_k$ and extends directly to the Gram system with multiple right--hand sides defining $\beta_k$.
\end{remark}

\subsection{Forward Gauss--Seidel iteration}
\label{FGS}

Let $W\in\mathbb{R}^{s\times s}$ be SPD, and let
\[
D=\operatorname{diag}(W).
\]
Since $W$ is SPD, the diagonal entries of $D$ are strictly positive.
We introduce the diagonally normalized Gram matrix
\[
W_c=D^{-1/2}WD^{-1/2}.
\]
The original reduced system
\[
W\alpha=m
\]
is equivalent to
\[
W_c\widetilde{\alpha}=\widetilde{m},
\qquad
\widetilde{\alpha}=D^{1/2}\alpha,
\qquad
\widetilde{m}=D^{-1/2}m.
\]
Since $W_c$ has unit diagonal, it can be split as
\[
W_c=I+L+L^\top,
\]
where $L$ denotes the strictly lower triangular part of $W_c$.
The FGS iteration applied to the normalized system computes
\begin{equation}
\label{eq:fgs-iteration}
(I+L)\widetilde{\alpha}^{(\nu+1)}
=
\widetilde{m}-L^\top\widetilde{\alpha}^{(\nu)}.
\end{equation}
Equivalently,
\[
\widetilde{\alpha}^{(\nu+1)}
=
(I+L)^{-1}\widetilde{m}
-
(I+L)^{-1}L^\top\widetilde{\alpha}^{(\nu)}.
\]
The corresponding iteration matrix is therefore
\[
G=-(I+L)^{-1}L^\top.
\]
This normalized formulation is algebraically equivalent to applying
FGS directly to $W\alpha=m$ using the splitting
\[
W=D+L_W+L_W^\top,
\]
where $L_W$ is the strictly lower triangular part of $W$. Indeed,
$L=D^{-1/2}L_WD^{-1/2}$.
\begin{proposition}[Residual after $\nu$ FGS sweeps]
\label{prop:fgs-residual}
Let
\[
W_c=I+L+L^\top,
\]
where $L$ is strictly lower triangular. Let
$\widetilde{\alpha}^{(0)}=0$, and let
$\widetilde{\alpha}^{(\nu)}$ be obtained after $\nu$ forward
Gauss--Seidel sweeps applied to
\[
W_c\widetilde{\alpha}=\widetilde{m}.
\]
Then the residual of the normalized system satisfies
\begin{equation}
\label{eq:fgs-residual-formula}
\widetilde{\rho}^{(\nu)}
:=
\widetilde{m}
-
W_c\widetilde{\alpha}^{(\nu)}
=
(-1)^\nu
\left(
L^\top(I+L)^{-1}
\right)^\nu
\widetilde{m}.
\end{equation}
Consequently,
\begin{equation}
\label{eq:fgs-residual-bound}
\left\|
\widetilde{\rho}^{(\nu)}
\right\|
\leq
\left\|
L^\top(I+L)^{-1}
\right\|^\nu
\left\|
\widetilde{m}
\right\|.
\end{equation}
\end{proposition}
\begin{proof}
The proof proceeds by induction. For $\nu=1$, the FGS iteration with
$\widetilde{\alpha}^{(0)}=0$ gives
\[
(I+L)\widetilde{\alpha}^{(1)}=\widetilde{m},
\]
and hence
\[
\widetilde{\alpha}^{(1)}
=
(I+L)^{-1}\widetilde{m}.
\]
Therefore,
\begin{align*}
\widetilde{\rho}^{(1)}
&=
\widetilde{m}
-
W_c\widetilde{\alpha}^{(1)}
\\
&=
\widetilde{m}
-
(I+L+L^\top)(I+L)^{-1}\widetilde{m}
\\
&=
-L^\top(I+L)^{-1}\widetilde{m},
\end{align*}
which proves~\eqref{eq:fgs-residual-formula} for $\nu=1$.
Assume that~\eqref{eq:fgs-residual-formula} holds for some $\nu\geq1$.
From the FGS update~\eqref{eq:fgs-iteration},
\[
(I+L)\widetilde{\alpha}^{(\nu+1)}
=
\widetilde{m}
-
L^\top\widetilde{\alpha}^{(\nu)}.
\]
Using
\[
(I+L+L^\top)(I+L)^{-1}
=
I+L^\top(I+L)^{-1},
\]
we obtain
\begin{align*}
\widetilde{\rho}^{(\nu+1)}
&=
\widetilde{m}
-
W_c\widetilde{\alpha}^{(\nu+1)}
\\
&=
\widetilde{m}
-
(I+L+L^\top)(I+L)^{-1}
\left(
\widetilde{m}
-
L^\top\widetilde{\alpha}^{(\nu)}
\right)
\\
&=
-L^\top(I+L)^{-1}
\left(
\widetilde{m}
-
(I+L+L^\top)\widetilde{\alpha}^{(\nu)}
\right)
\\
&=
-L^\top(I+L)^{-1}
\widetilde{\rho}^{(\nu)}.
\end{align*}
By the induction hypothesis,
\[
\widetilde{\rho}^{(\nu)}
=
(-1)^\nu
\left(
L^\top(I+L)^{-1}
\right)^\nu
\widetilde{m},
\]
and therefore
\[
\widetilde{\rho}^{(\nu+1)}
=
(-1)^{\nu+1}
\left(
L^\top(I+L)^{-1}
\right)^{\nu+1}
\widetilde{m}.
\]
This completes the induction. The bound
in~\eqref{eq:fgs-residual-bound} follows from the sub-multiplicativity
of the norm.
\end{proof}
Since
\[
\widetilde{\rho}^{(\nu)}
=
D^{-1/2}
\left(
m-W\alpha^{(\nu)}
\right),
\]
the bound in Proposition~\ref{prop:fgs-residual} controls the residual
of the original Gram system in the diagonally weighted norm
\[
\|v\|_{D^{-1}}
:=
\|D^{-1/2}v\|_2.
\]
\begin{corollary}[Asymptotic residual reduction rate]
\label{cor:spectral_radius}
Let
\[
H:=-L^\top(I+L)^{-1},
\qquad
\varrho:=\rho(H),
\]
where $H$ is the residual iteration matrix and $\rho(H)$ denotes its
spectral radius. Assume that $\widetilde m\neq0$ and that
$\widetilde{\alpha}^{(0)}=0$. Then, for every $\varepsilon>0$, there
exists a constant
\[
C=C(\varepsilon,W_c)>0,
\]
independent of $\nu$ and $\widetilde m$, such that
\begin{equation}
\label{eq:spectral_bound}
\frac{\|\widetilde{\rho}^{(\nu)}\|_2}
     {\|\widetilde m\|_2}
\leq
C(\varrho+\varepsilon)^\nu,
\qquad
\nu\geq1.
\end{equation}
More generally, the FGS iteration converges for every right-hand side
and every initial guess if and only if $\varrho<1$. For the residual
sequence generated from an arbitrary initial residual
$\widetilde{\rho}^{(0)}$, one has
\begin{equation}
\label{eq:asymptotic-residual-rate}
\limsup_{\nu\to\infty}
\|\widetilde{\rho}^{(\nu)}\|_2^{1/\nu}
\leq
\varrho.
\end{equation}
\end{corollary}
\begin{remark}
Since $W_c$ is SPD, the Gauss--Seidel iteration is convergent and
therefore $\varrho<1$. Its convergence rate nevertheless depends on
the off-diagonal structure of $W_c$ through the residual iteration
matrix $H$. The
polynomial basis used to construct the Krylov block influences this
off-diagonal structure. The following subsection derives an exact
moment representation for the Gram matrix associated with a raw
Chebyshev block. This representation provides a structural
interpretation of the observed behavior, but it does not by itself
guarantee that $\varrho$ remains uniformly bounded away from one.
The effectiveness of a fixed number of FGS sweeps is therefore also
assessed empirically in Section~\ref{sec:gram-experiments}.
\end{remark}

\subsubsection{Structure of Chebyshev Gram matrices}
\label{sec:cheb-gram-structure}

In this subsection, we analyze the Gram matrices associated with a
raw Chebyshev Krylov block in the unpreconditioned setting, i.e.,
with $M=I$. This setting isolates the effect of the polynomial basis
from that of the preconditioner. The analysis describes the basis-induced structure of the reduced
Gram systems and provides a qualitative interpretation of their conditioning.
In PCG--S, the blocks $Q_k$ generated after the first outer iteration
also contain corrections inherited from previous blocks. Therefore,
the moment representation derived below does not apply directly to
every matrix $Q_k^\top A Q_k$ generated by the algorithm. Rather, it
provides a structural explanation of the effect of using a Chebyshev
basis for the underlying raw Krylov block.
Let $[\lambda_{\min},\lambda_{\max}]$ be a scaling interval containing
the spectrum of the SPD matrix $A$, and define
\[
\widehat A
=
\frac{
2A-(\lambda_{\max}+\lambda_{\min})I
}{
\lambda_{\max}-\lambda_{\min}
}.
\]
The spectrum of $\widehat A$ is contained in $[-1,1]$. For a nonzero
vector $r_0$, define the Chebyshev Krylov block
\[
S_{s}
:=
\bigl[
T_0(\widehat A)r_0,\,
T_1(\widehat A)r_0,\,
\dots,\,
T_{s-1}(\widehat A)r_0
\bigr],
\]
where $T_j$ denotes the Chebyshev polynomial of the first kind of
degree $j$.
Since $A$ is SPD, it admits an orthogonal eigendecomposition $
A=V\Lambda V^\top,$ with
$\Lambda=\mathrm{diag}(\lambda_1,\dots,\lambda_n),$
where $0<\lambda_1\leq\cdots\leq\lambda_n.$
Let $v_\ell$ denote the $\ell$th column of $V$, and define
$
c:=V^\top r_0,$ with $c_\ell=v_\ell^\top r_0.$
The transformed eigenvalues are $
\widehat{\lambda}_\ell
=(
2\lambda_\ell-\lambda_{\max}-\lambda_{\min})/(
\lambda_{\max}-\lambda_{\min})
\in[-1,1].$
We consider both the standard-inner-product Gram matrix
\[
W^{(0)}
:=
S_{s}^\top S_{s}
\]
and the $A$--inner-product Gram matrix
\[
W^{(A)}
:=
S_{s}^\top A S_{s}.
\]
The latter is the form directly related to the reduced Gram matrices
arising in PCG--S.
The entries of these matrices are
\begin{align}
W^{(0)}_{ij}
&=
\sum_{\ell=1}^n
|c_\ell|^2
T_{i-1}(\widehat{\lambda}_\ell)
T_{j-1}(\widehat{\lambda}_\ell),
\label{eq:gram-entry}
\\
W^{(A)}_{ij}
&=
\sum_{\ell=1}^n
\lambda_\ell|c_\ell|^2
T_{i-1}(\widehat{\lambda}_\ell)
T_{j-1}(\widehat{\lambda}_\ell),
\label{eq:A-gram-entry}
\end{align}
for $1\leq i,j\leq s$.\\
For $x\in[-1,1]$ and nonnegative integers $m$ and $n$, the classical Chebyshev product identity is
\begin{equation}
\label{eq:cheb-product}
T_m(x)T_n(x)
=
\frac{1}{2}
\left(
T_{m+n}(x)+T_{|m-n|}(x)
\right).
\end{equation}
For $p\geq0$, define the moments
\[
\mu_p^{(0)}
:=
\sum_{\ell=1}^n
|c_\ell|^2T_p(\widehat{\lambda}_\ell)
\]
and
\[
\mu_p^{(A)}
:=
\sum_{\ell=1}^n
\lambda_\ell|c_\ell|^2
T_p(\widehat{\lambda}_\ell).
\]
\begin{theorem}[Chebyshev Gram matrix structure]
\label{thm:gram-structure}
For $\star\in\{0,A\}$, the entries of the corresponding Gram matrix
satisfy
\begin{equation}
\label{eq:gram-moments}
W^{(\star)}_{ij}
=
\frac{1}{2}
\left(
\mu_{i+j-2}^{(\star)}
+
\mu_{|i-j|}^{(\star)}
\right),
\qquad
1\leq i,j\leq s.
\end{equation}
Moreover, both $W^{(0)}$ and $W^{(A)}$ are symmetric positive
semidefinite. If $S_{s}$ has full column rank, then both matrices
are SPD.
\end{theorem}
\begin{proof}
Equation~\eqref{eq:gram-moments} follows by applying the identity~\eqref{eq:cheb-product} to
equations~\eqref{eq:gram-entry} and~\eqref{eq:A-gram-entry}. In the
standard inner product,
\begin{align*}
W^{(0)}_{ij}
&=
\sum_{\ell=1}^n
|c_\ell|^2
T_{i-1}(\widehat{\lambda}_\ell)
T_{j-1}(\widehat{\lambda}_\ell)
\\
&=
\frac{1}{2}
\sum_{\ell=1}^n
|c_\ell|^2
\left[
T_{i+j-2}(\widehat{\lambda}_\ell)
+
T_{|i-j|}(\widehat{\lambda}_\ell)
\right]
\\
&=
\frac{1}{2}
\left(
\mu_{i+j-2}^{(0)}
+
\mu_{|i-j|}^{(0)}
\right).
\end{align*}
The proof for $W^{(A)}$ is identical after replacing
$|c_\ell|^2$ with $\lambda_\ell|c_\ell|^2$.
Positive semidefiniteness follows from
\[
W^{(0)}=S_{s}^\top S_{s}
\]
and
\[
W^{(A)}
=
(A^{1/2}S_{s})^\top(A^{1/2}S_{s}).
\]
Since $A^{1/2}$ is nonsingular, both matrices are positive definite
whenever $S_{s}$ has full column rank.
\end{proof}
\begin{corollary}[Diagonal entries]
\label{cor:diagonal}
For $\star\in\{0,A\}$, the diagonal entries satisfy
\[
W^{(\star)}_{ii}
=
\frac{1}{2}
\left(
\mu_0^{(\star)}
+
\mu_{2i-2}^{(\star)}
\right),
\qquad
1\leq i\leq s.
\]
In particular,
\[
\mu_0^{(0)}
=
\|r_0\|_2^2
\]
and
\[
\mu_0^{(A)}
=
r_0^\top A r_0
=
\|r_0\|_A^2.
\]
\end{corollary}
Theorem~\ref{thm:gram-structure} shows that the entries of each
Chebyshev Gram matrix are completely determined by two structured
contributions from the corresponding moment sequence.  The term $\mu_{|i-j|}^{(\star)}$ has a Toeplitz-like
structure, whereas $\mu_{i+j-2}^{(\star)}$ has a Hankel-like
structure. Thus, the Gram matrix is the sum of a Toeplitz-like and a
Hankel-like component. Cancellation of higher-order moments may
reduce many off-diagonal entries, although such cancellation is not
guaranteed for an arbitrary finite-dimensional spectral measure.
\begin{corollary}[Normalized Chebyshev Gram matrix]
\label{cor:structurecheby}
Let $\star\in\{0,A\}$ and assume that $S_{s}$ has full column rank.
Define
\[
D^{(\star)}
:=
\mathrm{diag}\bigl(W^{(\star)}\bigr)
\]
and
\[
W_c^{(\star)}
:=
\bigl(D^{(\star)}\bigr)^{-1/2}
W^{(\star)}
\bigl(D^{(\star)}\bigr)^{-1/2}.
\]
Then $W_c^{(\star)}$ is SPD, has unit diagonal, and can be written as
\[
W_c^{(\star)}
=
I+L^{(\star)}+\bigl(L^{(\star)}\bigr)^\top,
\]
where $L^{(\star)}$ is strictly lower triangular.
Define the normalized moments
\[
\widetilde{\mu}_p^{(\star)}
:=
\frac{\mu_p^{(\star)}}{\mu_0^{(\star)}}.
\]
Then
\begin{equation}
\label{eq:normalized-gram-entry}
\bigl(W_c^{(\star)}\bigr)_{ij}
=
\frac{
\widetilde{\mu}_{i+j-2}^{(\star)}
+
\widetilde{\mu}_{|i-j|}^{(\star)}
}{
\sqrt{
\left(
1+\widetilde{\mu}_{2i-2}^{(\star)}
\right)
\left(
1+\widetilde{\mu}_{2j-2}^{(\star)}
\right)
}
},
\qquad
1\leq i,j\leq s.
\end{equation}
Moreover,
\[
\left|
\bigl(W_c^{(\star)}\bigr)_{ij}
\right|
\leq1
\]
for every $i$ and $j$.
\end{corollary}
\begin{proof}
By Theorem~\ref{thm:gram-structure} and
Corollary~\ref{cor:diagonal},
\[
W^{(\star)}_{ij}
=
\frac{\mu_0^{(\star)}}{2}
\left(
\widetilde{\mu}_{i+j-2}^{(\star)}
+
\widetilde{\mu}_{|i-j|}^{(\star)}
\right)
\]
and
\[
W^{(\star)}_{ii}
=
\frac{\mu_0^{(\star)}}{2}
\left(
1+\widetilde{\mu}_{2i-2}^{(\star)}
\right).
\]
Therefore,
\[
\bigl(W_c^{(\star)}\bigr)_{ij}
=
\frac{
W^{(\star)}_{ij}
}{
\sqrt{
W^{(\star)}_{ii}W^{(\star)}_{jj}
}
},
\]
which gives Equation~\eqref{eq:normalized-gram-entry}. Since
$W^{(\star)}$ is a Gram matrix, the Cauchy--Schwarz inequality gives
\[
\left|W^{(\star)}_{ij}\right|
\leq
\sqrt{
W^{(\star)}_{ii}W^{(\star)}_{jj}
}.
\]
The asserted bound follows.
\end{proof}
\begin{remark}[Interpretation and limitations]
\label{remark:moment-decay}
Since $|T_p(x)|\leq1$ for $x\in[-1,1]$, the moments defined above are
uniformly bounded. For a fixed finite-dimensional problem, however,
they need not decay with $p$ and may exhibit persistent oscillations.
Therefore, equation~\eqref{eq:gram-moments} does not by itself imply
off-diagonal decay, diagonal dominance, or a uniform condition-number
bound. It nevertheless shows how the polynomial basis, the
eigenvalue distribution, and the components of the residual in the
eigenvector basis determine the Gram-matrix entries, providing a
structural interpretation of the empirical behavior examined in
Section~\ref{sec:gram-experiments}.
Although the explicit derivation above considers $M=I$, the
left-preconditioned case can be analyzed through the symmetric
representative $
\mathcal{B}:=M^{-1/2}AM^{-1/2}$
of $\mathcal{A}=M^{-1}A$. This transformation is used only for
analysis; the implemented PCG--S method remains left preconditioned.
If $z_j=T_{j-1}(\widehat{\mathcal A})M^{-1}r$ and
$y_j=M^{1/2}z_j$, then $y_j=T_{j-1}(\widehat{\mathcal B})M^{-1/2}r,$ and
$z_i^\top A z_j=y_i^\top\mathcal B y_j.$
Hence, an analogous moment representation, with the same
Toeplitz-plus-Hankel structure, holds for the $A$--Gram matrix of the
left-preconditioned raw block. No general improvement in conditioning
follows from this transformation alone.
\end{remark}

\subsection{Connection between FGS and MGS}
\label{sec:gs-mgs}

The behavior of FGS applied to a Gram system can be also interpreted
through its connection with the sequential projection operations
used in MGS. The following result makes this
connection explicit for a particular right-hand side.

\begin{theorem}[Ruhe’s equivalence {\cite{Ruhe83}}]
\label{thm:ruhe}
Let $P=[p_1,\dots,p_s]\in\mathbb{R}^{n\times s}$
have linearly independent, column-normalized vectors, and let $
W=P^\top P=I+L+L^\top,$
where $L$ is strictly lower triangular.

Consider one forward Gauss--Seidel sweep, starting from
$\alpha^{(0)}=0$, applied to
$W\alpha=e_1.$
Then, $
(I+L)\alpha^{(1)}=e_1.$

Define a sequence of vectors by
\[
v^{(1)}=p_1
\]
and, for $j=2,\dots,s$,
\[
h_j=\langle p_j,v^{(j-1)}\rangle,
\qquad
v^{(j)}=v^{(j-1)}-h_jp_j.
\]
Then
\[
P\alpha^{(1)}=v^{(s)}.
\]
Thus, one FGS sweep is algebraically equivalent to an MGS-type sequential projection sweep of $p_1$ against
$p_2,\dots,p_s$.
\end{theorem}

\begin{proof}
Forward substitution in
\[
(I+L)\alpha^{(1)}=e_1
\]
gives
\[
\alpha_1^{(1)}=1
\]
and
\[
\alpha_j^{(1)}
=
-\sum_{i=1}^{j-1}
\langle p_j,p_i\rangle\alpha_i^{(1)},
\qquad
j=2,\dots,s.
\]

We prove by induction that
\[
v^{(j)}
=
\sum_{i=1}^{j}
\alpha_i^{(1)}p_i.
\]
For $j=1$, this follows from
$v^{(1)}=p_1$ and $\alpha_1^{(1)}=1$. Suppose that
\[
v^{(j-1)}
=
\sum_{i=1}^{j-1}
\alpha_i^{(1)}p_i.
\]
Then
\[
h_j
=
\langle p_j,v^{(j-1)}\rangle
=
\sum_{i=1}^{j-1}\langle p_j,p_i\rangle\alpha_i^{(1)}
=
-\alpha_j^{(1)}.
\]
Consequently,
\[
v^{(j)}
=
v^{(j-1)}-h_jp_j
=
\sum_{i=1}^{j-1}\alpha_i^{(1)}p_i+\alpha_j^{(1)}p_j
=
\sum_{i=1}^{j}\alpha_i^{(1)}p_i.
\]
Taking $j=s$ gives
\[
v^{(s)}=P\alpha^{(1)}.
\]
\end{proof}

\paragraph{Geometric interpretation and repeated sweeps.}

For the particular system $W\alpha=e_1$, the first FGS sweep admits
the sequential-projection interpretation described above. Each
projection removes the component of the current vector along one
column of $P$. Since these columns need not be mutually orthogonal, a
subsequent projection may reintroduce a component along a previously
processed column. Thus, one sweep does not generally produce a vector
orthogonal to all of $p_2,\dots,p_s$.

Subsequent FGS sweeps are more naturally interpreted as
residual-correction steps for the Gram system. Since $W$ is SPD, the
iterates converge to $\alpha^\star=W^{-1}e_1$, and the associated
vectors $P\alpha^{(\nu)}$ converge to
$P\alpha^\star\in\operatorname{range}(P)$. Since
\[
P^\top P\alpha^\star=e_1,
\]
the limiting vector is orthogonal to $p_2,\dots,p_s$. 

\paragraph{Extension to the $A$--inner product.}

Let the columns of
$Q=[q_1,\dots,q_s]\in\mathbb{R}^{n\times s}$ be normalized in the
$A$--inner product, so that
\[
q_j^\top A q_j=1.
\]
Then
\[
W=Q^\top A Q=I+L+L^\top,
\]
where $L$ is strictly lower triangular.

One FGS sweep applied to
\[
W\alpha=e_1,
\qquad
\alpha^{(0)}=0,
\]
is algebraically equivalent to the sequential projection process
\[
v^{(1)}=q_1,
\]
\[
h_j=q_j^\top A v^{(j-1)},
\qquad
v^{(j)}=v^{(j-1)}-h_jq_j,
\qquad
j=2,\dots,s.
\]
The resulting vector satisfies
\[
Q\alpha^{(1)}=v^{(s)}.
\]
Thus, for this particular right-hand side, one FGS sweep is
equivalent to one MGS-type sequential projection sweep in the
$A$--inner product.

The following standard identity relates the conditioning parameter
arising in finite-precision analyses of MGS in the $A$--inner product
to the condition number of the reduced Gram system.

\begin{lemma}[Gram and basis condition numbers]
\label{lem:gram-kappa}
Let $Q\in\mathbb{R}^{n\times s}$ have full column rank, let
$A\in\mathbb{R}^{n\times n}$ be SPD, and define
\[
W=Q^\top A Q.
\]
Then
\[
\kappa_2(A^{1/2}Q)^2=\kappa_2(W).
\]
\end{lemma}

\begin{proof}
Let $F=A^{1/2}Q$. Then $W=F^\top F$. Since the eigenvalues
of $W$ are the squares of the singular values of $F$, it follows that
\[
\kappa_2(W)
=
\frac{\lambda_{\max}(W)}{\lambda_{\min}(W)}
=
\frac{\sigma_{\max}(F)^2}{\sigma_{\min}(F)^2}
=
\kappa_2(F)^2
=
\kappa_2(A^{1/2}Q)^2.
\]
\end{proof}
\paragraph{Implications for the reduced Gram solves.}
Theorem~\ref{thm:ruhe} establishes an exact algebraic connection
between one FGS sweep and a sequential MGS-type projection sweep for
the particular right-hand side $e_1$. The reduced systems arising in
PCG--S instead have the more general right-hand sides
$m_k=Q_k^\top r^{(k)}$
and $B_k=-Q_k^\top AZ_{k+1}$
for the solution and block-direction updates, respectively. For these
right-hand sides, the specific sequential-projection interpretation
established for $e_1$ does not apply directly. Nevertheless, it exploits the
triangular structure of the Gram matrix in a manner analogous to
sequential projection operations, thereby providing an algebraic
motivation for its use as an inexact solver for the reduced systems.
Classical finite-precision analyses of MGS in non-standard inner
products~\cite{thomas1991congressus,thomas1992congressus,
thomas1992conpar,rozloznik2012,lowery2014}
show that the numerical behavior of the projection process depends on
the conditioning of the transformed basis $A^{1/2}Q$. By
Lemma~\ref{lem:gram-kappa},
$\kappa_2(A^{1/2}Q)=\kappa_2(Q^\top A Q)^{1/2}=
\kappa_2(W)^{1/2}.$
Thus, the conditioning parameter appearing in the MGS analyses is
directly related to the condition number of the reduced Gram matrix.
These analyses provide relevant context for understanding the role of
basis conditioning in Gram-based projection processes, but they do
not directly yield finite-precision stability bounds for the general
FGS solves used in PCG--S.
Similarly, the conditioning results of Philippe and
Reichel~\cite{Reichel2012} for raw Chebyshev Krylov bases motivate the
choice of the polynomial basis. Their assumptions, however, do not
apply directly to all the corrected search-direction blocks generated
by PCG--S, and hence their results do not establish a uniform
conditioning bound for the reduced Gram matrices considered here.
The accuracy of the FGS solves is instead described by the residual
bounds in Section~\ref{FGS} and assessed numerically in
Section~\ref{sec:gram-experiments}. Overall, the analysis identifies
FGS as a natural factorization-free and accuracy-adjustable solver for
the reduced Gram systems: its triangular splitting acts directly on
the correlations among the basis vectors, while the residual bounds
provide a practical mechanism for controlling the accuracy required
by the outer PCG--S iteration.

\subsection{Numerical Experiments}
\label{sec:gram-experiments}

In this section, we assess experimentally the use of an inexact FGS
solver for the reduced Gram systems arising in PCG--S. The benchmark
consists of linear systems obtained from a 27--point finite-difference
discretization of the three-dimensional Poisson equation on a cube
with Dirichlet boundary conditions. The experiments were performed
on $64$ GPUs, corresponding to $16$ nodes connected by HDR
InfiniBand, with four GPUs per node. The global problem size is $64\times200^3=5.12\times10^8$ degrees of freedom.

The objective is to determine whether a fixed number of FGS sweeps
provides sufficiently accurate reduced solves without adversely
affecting the observed outer convergence. We examine the empirical
structure of the Gram matrices generated during the iteration, the
relative residuals of the reduced solves
(Figure~\ref{fig:resestimate_empirical}), and the resulting PCG--S
convergence histories
(Figure~\ref{fig:pcgsvsfgs_empirical}).

Figure~\ref{fig:gram_patterns} shows the absolute values of
representative diagonally normalized Gram matrices for $s=10$ at
successive outer iterations. Although the PCG--S blocks $Q_k$ include
corrections inherited from previous outer iterations, the resulting
Gram matrices preserve a clear concentration around the diagonal.
This behavior is qualitatively consistent with the moment-based
structure derived in Section~\ref{sec:cheb-gram-structure} for an
unpreconditioned raw Chebyshev block and indicates that its structural
features remain relevant for interpreting the updated blocks
generated by the complete algorithm.

\begin{figure}[!htbp]
\centering
\includegraphics[width=0.7\textwidth]{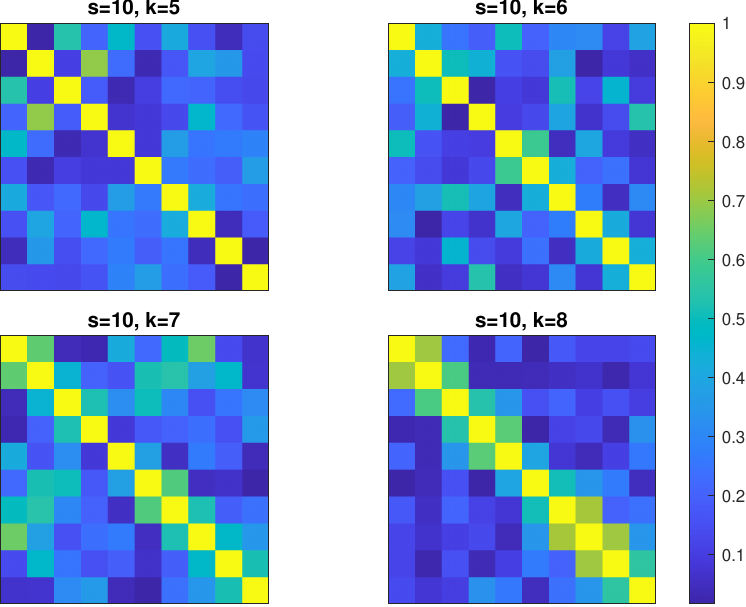}
\caption{Absolute values of representative diagonally normalized Gram
matrices with $s=10$.
The panels correspond to outer iterations $k=5,\ldots,8$ of PCG--S.
Problem: 27--point Poisson equation with
$5.12\times10^8$ degrees of freedom on $64$ GPUs.}
\label{fig:gram_patterns}
\end{figure}

At each outer iteration $k$, we measure the relative
residuals of the reduced systems for $\alpha_k$ and $\beta_k$ as
\[
\eta_{\alpha,k}
:=
\frac{\|Q_k^\top r^{(k)}-W_k\alpha_k\|_2}
     {\|Q_k^\top r^{(k)}\|_2},
\qquad
\eta_{\beta,k}
:=
\frac{\|-Q_k^\top AZ_{k+1}-W_k\beta_k\|_F}
     {\|Q_k^\top AZ_{k+1}\|_F}.
\]
The Frobenius norm is used for the $\beta_k$ system because it has
multiple right-hand sides. 

\begin{figure}[!htbp]
\centering
\includegraphics[width=0.6\textwidth]{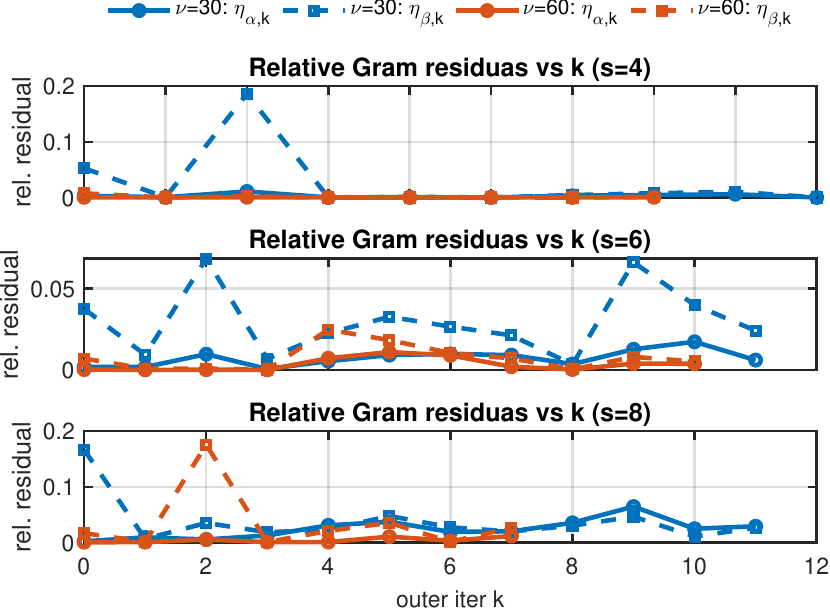}
\caption{Relative residuals $\eta_{\alpha,k}$ and
$\eta_{\beta,k}$ of the two reduced Gram systems versus the outer
iteration index $k$, for $\nu\in\{30,60\}$ FGS sweeps and step sizes
$s\in\{4,6,8\}$. Problem: 27--point Poisson equation with
$5.12\times10^8$ degrees of freedom on $64$ GPUs.}
\label{fig:resestimate_empirical}
\end{figure}

Figure~\ref{fig:resestimate_empirical}
reports these indicators versus the outer iteration index $k$ for
different numbers of FGS sweeps. For most
outer iterations, the residuals remain at the level of a few percent
or below. Larger transient values occur primarily for the
multiple-right-hand-side system defining $\beta_k$, with the largest
observed values remaining below approximately $2\times10^{-1}$.

Increasing the number of FGS sweeps from $\nu=30$ to $\nu=60$
generally improves the accuracy of the reduced solves, although the
reduction is not uniform at every outer iteration. This
nonmonotonicity is not unexpected, since each outer iteration involves
a different Gram matrix and different right-hand sides. In
particular, the curves should not be interpreted as the convergence
history of a single fixed reduced system.
However, no systematic growth of the reduced-system residuals is observed as
the outer iteration proceeds or as the step size increases over the
tested range. Together with the outer convergence histories in
Figure~\ref{fig:pcgsvsfgs_empirical}, these results show that the
accuracy delivered by a fixed, moderate number of FGS sweeps is
sufficient to preserve the observed PCG--S convergence behavior in
the configurations considered.

Figure~\ref{fig:pcgsvsfgs_empirical} reports the outer relative
residual histories. 
To obtain a meaningful convergence comparison, the PCG and PCG--S
residual histories are aligned by equivalent Krylov progress. More
precisely, the residual reported for PCG at the PCG--S outer iteration
$k$ is the residual obtained after $ks$ classical PCG iterations,
where $s$ is the block size of the corresponding PCG--S run. 

For the tested configurations with $\nu\geq30$, the convergence curves
obtained with FGS are nearly indistinguishable from those obtained
with the Cholesky-based Gram solve and closely follow the corresponding
classical PCG histories sampled after every $s$ iterations. This
provides empirical evidence that, within the tested range, a fixed
number of FGS sweeps is sufficient to avoid an appreciable degradation
of the outer convergence. The observed behavior is also compatible
with the practical inexact-solve guideline expressed by
condition~\eqref{cond:inexactKrylov2}, which, as noted above, is used
here as a heuristic rather than as a rigorous sufficient condition.

No numerical instability is observed over the tested range of step
sizes. The results indicate that the reduced Gram systems remain
sufficiently regular for the prescribed number of FGS sweeps to
provide adequate accuracy at the moderate values of $s$ considered
here. This behavior is qualitatively consistent with the structural
interpretation of the Chebyshev Gram matrices developed in
Section~\ref{sec:cheb-gram-structure}, but it does not constitute a
uniform condition-number or stability bound for all Gram matrices
generated by PCG--S.

\begin{figure}[!htbp]
\centering
\includegraphics[width=0.6\textwidth]{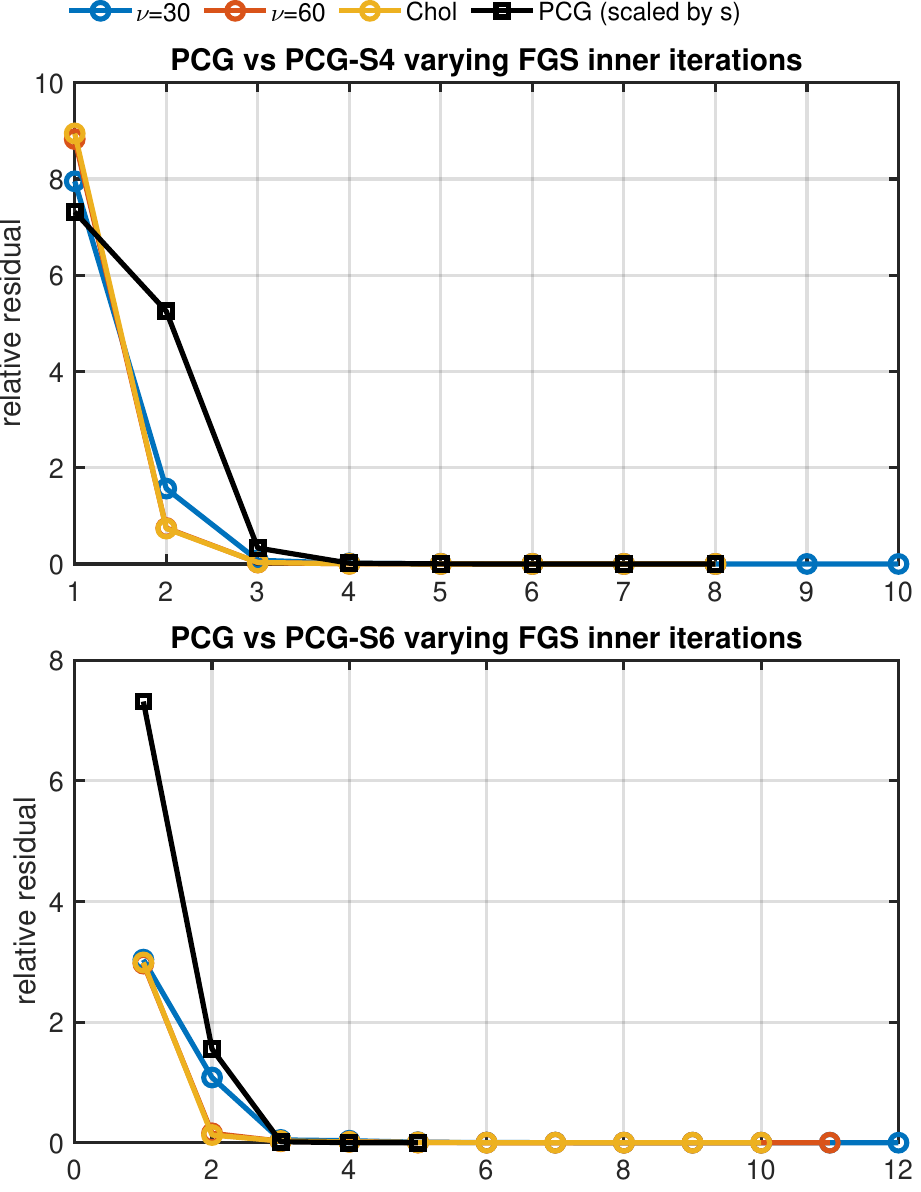}
\caption{Relative residual histories for PCG--S with two FGS sweep
counts $\nu$ and step sizes $s\in\{4,6\}$, compared with classical
PCG and PCG--S using a Cholesky-based Gram solve. At PCG--S outer
iteration $k$, the corresponding PCG residual is sampled after $ks$
classical iterations, so that the curves are aligned by equivalent
Krylov progress. Problem: 27--point Poisson equation with
$5.12\times10^8$ degrees of freedom on $64$ GPUs.}
\label{fig:pcgsvsfgs_empirical}
\end{figure}

\section{Practical Implementation}
\label{sec:impl}

The PCG-S method described in Algorithms~\ref{alg:pcgs} and~\ref{alg:MPK} has been implemented in the
\texttt{BootCMatchGX} framework, a distributed NVIDIA multi--GPU software environment for
large--scale sparse linear algebra. The implementation builds upon the existing \texttt{BootCMatchGX} infrastructure, using its CSR matrix representation for distributed sparse matrices and combining MPI-based domain decomposition and inter-process communication across nodes with CUDA-based acceleration on each GPU.

The Matrix Power Kernel (Algorithm~\ref{alg:MPK}), responsible for generating the Chebyshev
block basis, relies on a distributed SpMV
implementing explicit com\-munica\-tion--com\-pu\-ta\-tion overlap.
Given the row--wise domain decomposition, halo exchanges are required at each
SpMV application. In our implementation, non--blocking MPI communications are
initiated prior to launching the local GPU kernel. While boundary data are
being transferred, the interior rows are processed on the device. After
communication completion, boundary contributions are finalized.
This strategy follows the standard overlap paradigm for distributed SpMV and
is particularly effective in strong--scaling regimes where communication latency
dominates arithmetic cost. Within each MPK cycle, SpMV and
preconditioner applications alternate.

The current implementation provides an Algebraic MultiGrid (AMG) preconditioner
based on aggregation of unknowns, driven by the so-called {\em Compatible Weighted Matching} (see~\cite{Dambra2018,ieee2023} for details), already available in
\texttt{Boot\-C\-MatchGX}. However, the solver architecture exposes a generic
preconditioner interface. Any future preconditioner that conforms to this
interface can be seamlessly integrated into the PCG-S framework without
modifying Algorithms~\ref{alg:pcgs} or~\ref{alg:MPK}.
Thus, the $s$-step formulation remains agnostic to the specific choice of
preconditioner, and the communication-reduction strategy is orthogonal to
preconditioning design.

All block linear algebra operations arising in PCG--S are formulated in terms of small dense matrix kernels to maximize arithmetic intensity on GPUs.
In particular, local operations are performed using the GEMM and GEMV routines from the cuBLAS library. These kernels are used for the block operations in Algorithm~\ref{alg:pcgs}, including updates of the approximate solution and residual (lines 7--8), construction of the local Gram contributions and
right-hand sides (lines 5 and 13), and block recurrences (line 15).
Recasting vector recurrences as dense block operations enables efficient transparent
utilization of current tensor cores and high--throughput GPU pipelines, significantly
increasing computational efficiency relative to classical scalar CG updates.
All such operations are executed locally on the device, avoiding unnecessary
host-device transfers and increasing effective flop utilization.

The reduced Gram matrix $W \in \mathbb{R}^{s \times s}$ is assembled from products between local vector blocks.
Each process computes its local contribution via cuBLAS GEMM, and an MPI all--reduce combines these contributions to obtain the global matrix $W$, replicated on all ranks.
The FGS iterations used to approximate the solutions of the reduced
systems at lines~6 and~14 are performed redundantly on each CPU.
This design avoids additional communication and is justified by the
small dimension of the Gram systems ($s\leq 10$ in the configurations
considered) and by the negligible cost of the replicated FGS solves
relative to that of the global reductions.

After completion of the FGS iteration, the coefficient vectors and matrices
are transferred back to the GPU and used for block updates.
Overall, the implementation is explicitly designed to operate in the regime
where latency dominates arithmetic cost, thereby exploiting the asymptotic
advantages predicted by the strong and weak scaling analyses of the next Section.

\subsection{Performance Model}
\label{sec:costperiter}

In large--scale GPU--accelerated environments, the performance of Krylov subspace methods is increasingly dominated by communication costs rather than arithmetic throughput. A quantitative performance model is therefore essential to characterize the trade--off between reduced global synchronization and increased local computation introduced by the $s$--step formulation. In this section, we derive a latency-based model that captures these effects and predicts the scalability regime in which the proposed method becomes advantageous over classical PCG. For simplicity, the model omits the global reductions used solely to
evaluate the stopping criterion in both methods. If these reductions
are performed separately at every classical or outer iteration, as in our implementation, their
inclusion would further favor PCG--S, making the present
communication-savings estimate conservative.

Let \( W_{\mathrm{SpMV}} \) and \( W_{\mathrm{prec}} \) denote the cost of one parallel sparse matrix--vector
multiplication and one preconditioner application, respectively.
Let \( W_{\mathrm{axpy}} \), \( W_{\mathrm{dot}} \), and \( W_{\mathrm{allreduce}} \) be the costs of one vector update
(axpy), one dot product, and one global reduction (required by parallel dot products), respectively.

Then, the cost of a single iteration of the PCG method can be approximated as
\[
W_{\mathrm{PCGiter}} \approx
W_{\mathrm{SpMV}} + W_{\mathrm{prec}} + 3 W_{\mathrm{axpy}} + 2 W_{\mathrm{dot}} + 2 W_{\mathrm{allreduce}}.
\]

For the PCG-S method with a Chebyshev--based MPK kernel and
\(\nu\) FGS sweeps on the Gram systems,
the cost of one outer iteration is
\begin{align*}
W_{\mathrm{PCGSiter}}(\nu)
&\approx
s\bigl(W_{\mathrm{SpMV}}+2W_{\mathrm{axpy}}\bigr)
+sW_{\mathrm{prec}}
\\
&\quad
+\left(\frac{s(s+1)}{2}+s\right)W_{\mathrm{dot}}
+s^2W_{\mathrm{dot}}
+2W_{\mathrm{allreduce}}
\\
&\quad
+2W_{\mathrm{vecupdates}}
+2W_{\mathrm{matupdates}}
+W_{\mathrm{FGS}}(s,\nu),
\end{align*}
where \( W_{\mathrm{matupdates}} \) accounts for local rectangular matrix updates (line~15 in Algorithm~\ref{alg:pcgs}), $W_{\mathrm{vecupdates}}$ denotes the cost of one block
matrix--vector update, accounting for the solution and residual
updates (lines~7--8 of Algorithm~\ref{alg:pcgs}), and $W_{\mathrm{FGS}}(s,\nu)$ denotes the arithmetic work required
to apply $\nu$ FGS sweeps to both reduced systems. 

By grouping equivalent terms, we obtain
\begin{align*}
W_{\mathrm{PCGSiter}}(\nu)
&\approx
sW_{\mathrm{SpMV}}
+sW_{\mathrm{prec}}
+2sW_{\mathrm{axpy}}
+\frac{3}{2}(s^2+s)W_{\mathrm{dot}}
\\
&\quad
+2W_{\mathrm{allreduce}}
+2W_{\mathrm{vecupdates}}
+2W_{\mathrm{matupdates}}
+W_{\mathrm{FGS}}(s,\nu).
\end{align*}

If we explicitly account for the floating-point work associated with
the local vector and matrix operations, which do not require
inter-process communication, the costs of $s$ PCG iterations and one
PCG--S outer iteration, for a global problem of dimension $n$, can be
written as
\begin{align*}
sW_{\mathrm{PCGiter}}
&\approx
sW_{\mathrm{SpMV}}
+sW_{\mathrm{prec}}
+2sW_{\mathrm{allreduce}}
+8sn\ \mathrm{flops},
\\[4pt]
W_{\mathrm{PCGSiter}}(\nu)
&\approx
sW_{\mathrm{SpMV}}
+sW_{\mathrm{prec}}
+2W_{\mathrm{allreduce}}
\\
&\quad
+\left[
8sn
+\frac{3}{2}(s^2+s)n
+2(s^2+s)n
\right]\mathrm{flops}
+W_{\mathrm{FGS}}(s,\nu)
\\
&=
sW_{\mathrm{SpMV}}
+sW_{\mathrm{prec}}
+2W_{\mathrm{allreduce}}
\\
&\quad
+\left[
8sn+\frac{7}{2}(s^2+s)n
\right]\mathrm{flops}
+W_{\mathrm{FGS}}(s,\nu).
\end{align*}

Note that in our parallel implementation, FGS sweeps are replicated across all parallel tasks.
Therefore, their cost \( W_{\mathrm{FGS}}(s,\nu) \) represents a \emph{local replicated overhead}.
In contrast, the cost of parallel operations, i.e., vector and matrix updates which are executed concurrently by all tasks, increases quadratically with \(s\).

Let \( t_{\mathrm{flop}} \) denote the execution time of one double--precision flop on a single processor,
and let us adopt the standard latency--bandwidth communication model. In the
regime considered here, the global reductions involve relatively
small messages, while their cost is dominated by collective
synchronization latency. We therefore neglect the lower-order
bandwidth contribution and approximate the time of one MPI
all-reduce among $P$ processes as
\[
T_{\mathrm{allreduce}}
\approx
\alpha_{\mathrm{lat}}\log_2P.
\]
The iteration times on P processes can be expressed as
\begin{align*}
T_{s,\mathrm{PCGiter}}
&\approx
sT_{\mathrm{SpMV}}
+sT_{\mathrm{prec}}
+2s\alpha_{\mathrm{lat}}\log_2P
+\frac{8sn}{P}\,t_{\mathrm{flop}},
\\[4pt]
T_{\mathrm{PCGSiter}}(\nu)
&\approx
sT_{\mathrm{SpMV}}
+sT_{\mathrm{prec}}
+2\alpha_{\mathrm{lat}}\log_2P
\\
&\quad
+\frac{
8sn+\tfrac{7}{2}(s^2+s)n
}{P}\,t_{\mathrm{flop}}
+T_{\mathrm{FGS}}(s,\nu),
\end{align*}

where $T_{\mathrm{SpMV}}$ and $T_{\mathrm{prec}}$ are the execution times for one parallel SpMV and for one parallel application of a preconditioner, respectively, and $T_{\mathrm{FGS}}(s,\nu)$ is the cost for Gram solves.

\paragraph{Interpretation of the difference.}

\begin{align}
\Delta(P,s,\nu)
&=
T_{\mathrm{PCGSiter}}(\nu)
-
T_{s,\mathrm{PCGiter}}
\notag
\\
&=
2\alpha_{\mathrm{lat}}(1-s)\log_2P
+
\frac{7}{2}(s^2+s)\frac{n}{P}\,t_{\mathrm{flop}}
+
T_{\mathrm{FGS}}(s,\nu)
\label{eq:delta}
\end{align}
represents the difference between the time required for one outer iteration of the
PCG-S algorithm and \(s\) consecutive classical PCG iterations.
A negative value of \(\Delta\) indicates that one iteration of PCG-S is faster than $s$ iterations of the classical PCG.

\subparagraph{Communication term.}
The first component,
\[
2\alpha_{\mathrm{lat}} \,(1 - s) \log_2 P,
\]
quantifies the variation in global communication cost.
In the classic PCG method, each iteration involves two global reductions,
whereas PCG-S performs only two reductions per outer iteration, corresponding to
$s$ PCG iterations.
For \(s>1\), this term becomes \emph{negative}, indicating that PCG-S reduces the number of global synchronizations by a factor of $s$.
Consequently, this component favors PCG-S in large-scale parallel systems where latency dominates the overall execution time.

\subparagraph{Computational term.}
The second component,
\[
\frac{7}{2}(s^2+s)\frac{n}{P}\,t_{\mathrm{flop}}
\]
represents the difference in distributed floating-point workload.
It is positive for every $s\geq1$ and grows quadratically with $s$.
In the communication-avoiding regime $s>1$, it therefore represents
an additional local arithmetic cost relative to $s$ classical PCG
iterations. This cost arises primarily from the extra inner products
required to form the reduced Gram systems and from the block
recurrences within the $s$--step Krylov basis. Consequently, this
term increases $\Delta$ and penalizes PCG--S as $s$ grows.

\subparagraph{Gram solve overhead.}
The last component, denoted by $T_{\mathrm{FGS}}(s,\nu)$, captures
the overhead of applying $\nu$ FGS sweeps to both reduced Gram
systems. Since the Gram system
for $\alpha_k$ has one right-hand side whereas the Gram system for
$\beta_k$ has $s$ right-hand sides, 
for the parametric evaluation of the model, we approximate the FGS
time by its leading-order arithmetic cost,
\[
T_{\mathrm{FGS}}(s,\nu)
\approx
c_{\mathrm{FGS}}\nu s^3t_{\mathrm{flop}},
\]
where $c_{\mathrm{FGS}}$ is an implementation-dependent constant
that accounts for the detailed FGS operation count and for the
difference between the effective execution rates of the local
distributed kernels and the replicated Gram solves. In the
parametric results below, we set $c_{\mathrm{FGS}}=1$ to obtain a
simple leading-order model. This contribution is positive,
grows linearly with $\nu$ and cubically with $s$ in leading order,
and represents a replicated local overhead in the total runtime of
$s$--step PCG. 

\medskip
In summary, \(\Delta(P,s,\nu)\) highlights the trade--off between
\emph{reduced communication latency} (first term, beneficial)
and \emph{increased local computation} (second and third terms, detrimental).
As expected, PCG-S can become advantageous when
communication latency dominates over local computation, i.e., for large \(P\)
and high--latency parallel architectures.

\paragraph{Parametric interpretation under strong scaling.}

In the strong--scaling regime, the global problem size \(n\) is fixed while the number of processes \(P\) increases, so that each process handles a smaller local workload \(c(P)=n/P\).
Figure~\ref{fig:Strong_Delta_vsP_nu30} shows the behavior of
\(\Delta_{\mathrm{strong}}(P,s,\nu)\) for $\nu=30$, $n=500^3$.

As $P$ increases from $32$ to $512$, the distributed arithmetic term
in $\Delta_{\mathrm{strong}}$ decreases as $1/P$, whereas the
replicated FGS overhead is independent of $P$ in the present model.
Conversely, the communication term,
\[
2\alpha_{\mathrm{lat}}\log_2 P\,(1 - s),
\]
grows logarithmically with \(P\), since the cost of global reductions increases with the depth of the communication tree.
Because this term is \emph{negative} for \(s>1\), it represents a latency \emph{reduction} that becomes increasingly dominant at large process counts.

The combined effect produces a monotonic decrease in \(\Delta_{\mathrm{strong}}\) as \(P\) grows,
eventually leading to a crossover point where \(\Delta_{\mathrm{strong}}<0\),
and the PCG-S algorithm becomes more efficient than the classical PCG.
This transition occurs at larger process counts for larger \(s\),
as the enhanced communication savings of PCG-S increasingly compensate for its greater local computational cost.

Overall, the strong--scaling results confirm that for small \(P\) or small steps \(s\), the classical PCG remains competitive;
however, at large \(P\) and high $\alpha_{\mathrm{lat}}/t_{\mathrm{flop}}$ ratios,
the reduced synchronization cost suggests that PCG-S becomes increasingly advantageous.
\begin{figure}[!htbp]
  \centering
  \includegraphics[width=0.7\linewidth]{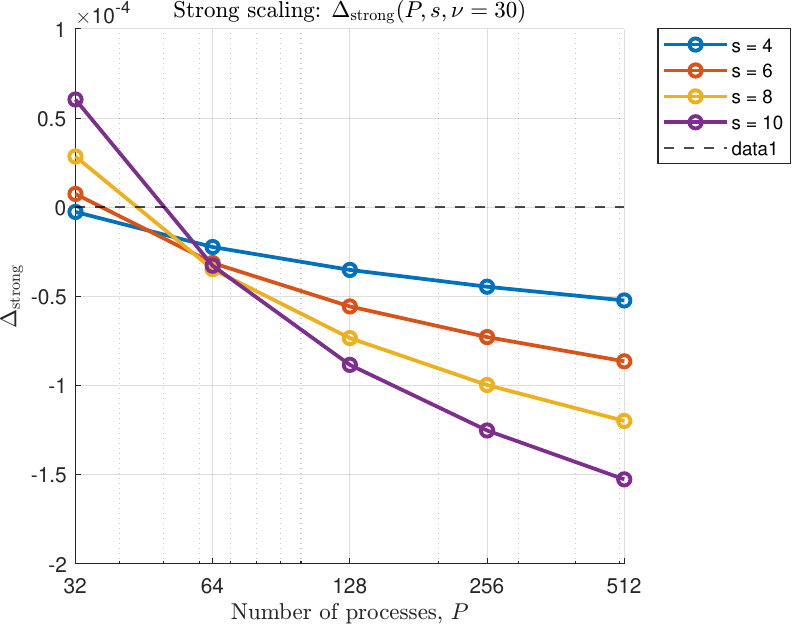}
  \caption{Strong--scaling behavior of
  \(\Delta_{\mathrm{strong}}(P,s,\nu=30)\) as a function of the number of processes \(P\)
  for several step sizes \(s\), with fixed \(n=500^3\). The dashed line marks \(\Delta_{\mathrm{strong}}=0\).}
  \label{fig:Strong_Delta_vsP_nu30}
\end{figure}

\paragraph{Parametric interpretation under weak scaling.}

Under weak scaling, where \(n = cP\) and each processor handles a fixed local problem size \(c\),
the difference in \eqref{eq:delta} becomes
\[
\Delta_{\mathrm{weak}}(P,s,\nu)
=
2\alpha_{\mathrm{lat}}(1-s)\log_2 P
+
\frac{7}{2}c(s^2+s)\,t_{\mathrm{flop}}
+
\nu s^3\,t_{\mathrm{flop}}.
\]
The sign of \(\Delta_{\mathrm{weak}}\) is determined by the competition between:

\begin{enumerate}
\item the \emph{communication gain} of the PCG-S method,
given by $2\alpha_{\mathrm{lat}}\,(1 - s)\log_2 P < 0$ for \(s > 1\),
which represents the latency reduction from fewer global synchronizations; and
\item the \emph{computational penalty} of PCG-S,
\(\tfrac{7}{2}c(s^2+s)\,t_{\mathrm{flop}}+ \nu s^3\,t_{\mathrm{flop}}\),
arising from the additional local arithmetic operations and sequential overhead.
\end{enumerate}

PCG-S becomes asymptotically advantageous
(\(\Delta_{\mathrm{weak}} < 0\))
if and only if the communication savings dominate the extra computation, that is,
\[
2\alpha_{\mathrm{lat}} \,(s - 1) \log_2 P
\;>\;
\frac{7}{2}c(s^2+s)\,t_{\mathrm{flop}}
+ \nu s^3t_{\mathrm{flop}}.
\]
Solving for $P$ yields the critical process count beyond which
PCG--S outperforms $s$ classical PCG iterations:
\begin{equation}
\label{eq:Pcrit}
\log_2 P_{\mathrm{crit}}(s,\nu)
=
\frac{
t_{\mathrm{flop}}
\left[
\frac{7}{2}c(s^2+s)
+
\nu s^3
\right]
}{
2\alpha_{\mathrm{lat}}(s-1)
},
\qquad s>1.
\end{equation}
\begin{table}[htbp]
\centering
\begin{tabular}{@{}ccc@{}}
\toprule
$s$
&
$\log_2 P_{\mathrm{crit}}(s,30)$
&
$P_{\mathrm{crit}}(s,30)$
\\
\midrule
$2$  & $8.40$  & $338$    \\
$3$  & $8.40$  & $338$    \\
$4$  & $9.33$  & $646$    \\
$5$  & $10.50$ & $1449$   \\
$6$  & $11.76$ & $3469$   \\
$7$  & $13.07$ & $8580$   \\
$8$  & $14.40$ & $21621$  \\
$9$  & $15.75$ & $55115$  \\
$10$ & $17.11$ & $141582$ \\
\bottomrule
\end{tabular}
\caption{Critical process count $P_{\mathrm{crit}}(s,30)$ under weak
scaling for $c=200^3$, $\alpha_{\mathrm{lat}}=10^{-6}\,\mathrm{s}$,
$t_{\mathrm{flop}}=10^{-13}\,\mathrm{s}$, and
$c_{\mathrm{FGS}}=1$. The process counts are rounded upward to the
nearest integer.}
\label{tab:pcrit}
\end{table}

The values in Table~\ref{tab:pcrit} represent the minimum number of processes
$P_{\mathrm{crit}}(s,30)$ required for the PCG-S method
to become faster than the classical PCG in the weak--scaling regime,
under the model's assumptions on $\alpha_{\mathrm{lat}}$ and $t_{\mathrm{flop}}$.
For $P > P_{\mathrm{crit}}$, the reduction in global communication
outweighs the additional local floating-point work,
leading to $\Delta < 0$ (PCG-S becomes faster).
As noted in Section~\ref{strong}, the model underestimates communication costs
relative to measured data; the actual crossover process counts are therefore
likely lower than (more favorable than) those in the table.

As $s$ increases, $\log_2 P_{\mathrm{crit}}$ grows rapidly, causing
the critical process count $P_{\mathrm{crit}}$ to increase
approximately exponentially over the range considered. For small
step sizes ($s=2$--$3$), the predicted crossover occurs at a few
hundred processes. For $s=4$, it remains below one thousand
processes, whereas for larger values of $s$ it shifts from thousands
to more than $10^5$ processes. Thus, as $s$ grows, the larger
communication savings become beneficial only at increasingly large
parallel scales because they must compensate for the additional
local arithmetic work.

In summary, the parametric model predicts that PCG--S can become
advantageous in high-latency, large-scale environments. At lower
process counts, or when the additional block arithmetic dominates
the communication savings, classical PCG remains competitive.

As also illustrated in Figure~\ref{fig:delta_vsP}, the communication
term becomes increasingly influential as the process count grows.
Consequently, a large latency-to-compute ratio
$\alpha_{\mathrm{lat}}/t_{\mathrm{flop}}$ favors PCG--S and lowers
the process count at which the predicted crossover occurs.
%
\begin{figure}[!htbp]
  \centering
  \includegraphics[width=0.7\textwidth]{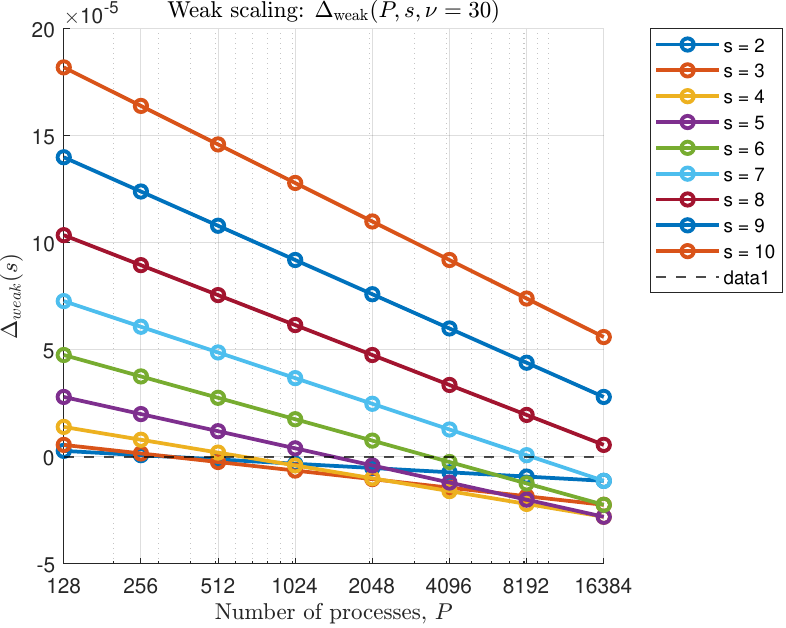}
  \caption{Weak--scaling behavior of \(\Delta_{\mathrm{weak}}(P,s,\nu=30)\) versus the number of processes \(P\),
  for several step sizes \(s\), with scaling factor \(c=200^3\). The dashed line marks \(\Delta_{\mathrm{weak}}=0\).
 }
  \label{fig:delta_vsP}
\end{figure}

\section{Related Work}
\label{sec:related}

Reducing global synchronization costs is a central challenge for Krylov subspace methods on modern large-scale architectures.

One class of approaches reduces the frequency of global synchronization. The $s$-step CG method, originally introduced by Chronopoulos and Gear~\cite{ChronGear1989,ChronGear1989p}, groups several consecutive Krylov iterations into a single outer iteration, reducing the number of global reductions by a factor of~$s$. This idea was later formalized within the broader framework of communication-avoiding (CA) Krylov methods by Hoemmen and Demmel~\cite{Hoemmen2010,Demmel2013CAKrylov}, where classical Krylov recurrences are reorganized so as to minimize synchronization while preserving the underlying Krylov sequence in exact arithmetic. A related but conceptually distinct strategy is followed by enlarged Krylov subspace methods~\cite{Grigori_enlarged1,Grigori_enlarged2}, which generate multiple search directions per iteration by expanding the approximation space, often through domain decomposition or multiple starting vectors.

A well--known limitation of early $s$--step and CA-CG formulations is their sensitivity to finite-precision effects. When monomial bases are used, the conditioning of the associated Gram matrices grows rapidly with~$s$, which restricts practical step sizes. Moreover, the Chronopoulos--Gear variant reduces the number of global reductions by introducing auxiliary recurrence relations for dot products and Gram quantities; while effective in reducing synchronization, these additional recurrences increase algebraic coupling among quantities and may amplify rounding errors. Several works therefore investigated strategies to improve the numerical robustness of reduced-synchronization Krylov methods. Hoemmen~\cite{Hoemmen2010} discussed the use of better--conditioned polynomial bases, such as Chebyshev or Newton polynomials, while Carson and collaborators studied adaptive $s$--step strategies and deflation or augmentation techniques for CA-CG~\cite{CarsonAdaptive2018,CarsonKnightDemmel2014}. Mixed-precision variants of $s$--step Lanczos and CG have also been analyzed~\cite{CarsonGergelitsYamazaki2022}, where higher precision is used in the Gram matrix formation to mitigate conditioning issues. Although effective numerically, approaches relying on quadruple precision are difficult to reconcile with the performance characteristics of modern accelerator--based architectures, where such arithmetic is typically emulated in software and introduces substantial computational overhead.

A complementary line of research aims instead to hide communication latency through pipelining. Pipelined CG methods reorganize the classical recurrences so that global reductions can be overlapped with SpMV and preconditioner applications~\cite{GhyselsVanroose2014}. While these methods significantly improve strong scaling, their multi--term recurrences may amplify rounding errors and reduce the maximal attainable accuracy in finite precision~\cite{CarsonPipelined2018}. Stabilized pipelined variants have therefore been proposed to mitigate error propagation while preserving communication--hiding properties~\cite{Cools2017Stabilized,Cornelis2019StableRecurrences}.

Beyond algorithmic formulations, several works have investigated scalable implementations of these ideas on distributed--memory systems. Pipelined CG variants for large--scale architectures overlap global reductions with SpMV and preconditioner applications, thereby reducing idle time caused by synchronization~\cite{tiwari2022pipepcg}. Other implementations combine pipelining with $s$--step formulations to simultaneously reduce and hide communication costs, overlapping the construction of block--Krylov bases with the evaluation of global reductions~\cite{Tiwari2021}. More general frameworks integrating pipelining with communication--avoiding $s$--step methods have also been proposed to improve scalability on large parallel systems~\cite{MG2025,mayer2026pspcg}.
The implementations considered in these latter studies use an
all-MPI execution model on distributed-memory CPU systems. Although
they provide important algorithmic reference points, their published
performance results do not constitute a like-for-like baseline for
the present multi-GPU implementation, because the relative costs of
local kernels, preconditioner applications, data movement, and global
reductions differ substantially between the two architectures.
To the best of our knowledge, no publicly available fully distributed
multi-GPU implementation of preconditioned $s$--step CG has previously
been assessed at the scale considered here.

In this respect, the present work complements the existing literature
by combining a numerically stabilized Chebyshev--based formulation, an
inexact FGS Gram solve, and an implementation explicitly designed for
modern multi-GPU architectures.

We adopt an $s$--step formulation that remains close to the classical PCG algorithm while avoiding auxiliary recurrence--based updates of the Gram matrix and long multi--term recurrences typical of pipelined approaches. The reduced systems are formed explicitly at each outer iteration, limiting error propagation while preserving the algebraic structure of classical PCG. Our approach uses a Chebyshev basis to mitigate the
basis-induced ill-conditioning of the reduced Gram matrices relative
to the usual monomial representation. The reduced systems are solved
approximately by FGS in standard double precision. The structural
analysis in Section~\ref{sec:erroranalysis} provides a qualitative
interpretation of the resulting Gram matrices, while the accuracy and
stability of the inexact solves for the finite-dimensional systems
considered here are assessed experimentally.
The theory of inexact Krylov methods~\cite{vandenEshofSleijpen2004,bouras2005}
motivates the control of the perturbations introduced by approximate
inner solves. In the present setting, the relative Gram-residual
criterion is used as a practical guideline rather than as a complete
problem-independent convergence guarantee.
Our approach is inspired by recent work on low-synchronization projected solves in GMRES~\cite{Thomas2024IGSGMRES}, where iterated Gauss--Seidel sweeps were shown to provide an effective alternative to dense factorizations in communication--avoiding formulations.

\section{Performance Results at Large Scale}
\label{sec:results}

We report results for our multi--GPU implementation of preconditioned $s$--step CG available in \texttt{BootCMatchGX}. Experiments were conducted on the Leonardo supercomputer (ranked 10th in the November 2025 Top500 list; BullSequana XH2000, Xeon Platinum 8358 32C 2.6 GHz, NVIDIA A100 SXM4 64 GB, quad-rail NVIDIA HDR100 InfiniBand) and on MareNostrum~5 at the Barcelona Supercomputing Center (accelerated partition ranked 14th in the November 2025 Top500 list; BullSequana XH3000, Intel Xeon Platinum 8460Y+ 32C 2.3 GHz, NVIDIA H100 SXM 64 GB, InfiniBand NDR), depending on resource availability.

The choice of linear system sizes and step sizes $s$ is guided by the theoretical model of Section~\ref{sec:costperiter} and by the maximum scale currently accessible on these systems. Further validation of the asymptotic predictions, including larger step sizes, will require benchmark--scale allocations planned for future work.

Strong and weak scalability are evaluated using linear systems arising from the standard 3D Poisson benchmark: matrix $A$ corresponds to the 27--point finite--difference discretization of the Poisson equation on the unit cube (as in HPCG), with Dirichlet boundary conditions and unit right--hand side. Matrices are row--partitioned across GPUs, mapping the 3D domain onto a 3D MPI process grid with one GPU per task.

\subsection{Strong Scalability}
\label{strong}

In this section we report strong--scaling results obtained on MareNostrum~5. The Gram systems are solved using $\nu = 30$ sweeps of the FGS method. Experiments are performed on a problem of size $500^3$ DOFs, using between 32 and 256 GPUs.

Our primary goal is to compare the $s$-step CG method with the classical CG algorithm. To isolate the impact of global communication per iteration and assess the effect of the $s$--step formulation, we consider the unpreconditioned case and fix the number of (outer) iterations to $k_{\max} = 20$. In the un--preconditioned setting, the preconditioner is formally set to the identity matrix, $M = I$, so that $M^{-1}$ reduces to the identity operator and no preconditioning is applied within the MPK algorithm.
Therefore we use the label $CG--S$ in the corresponding figures.

\begin{figure}[!htbp]
\centering
\includegraphics[width=0.7\textwidth]{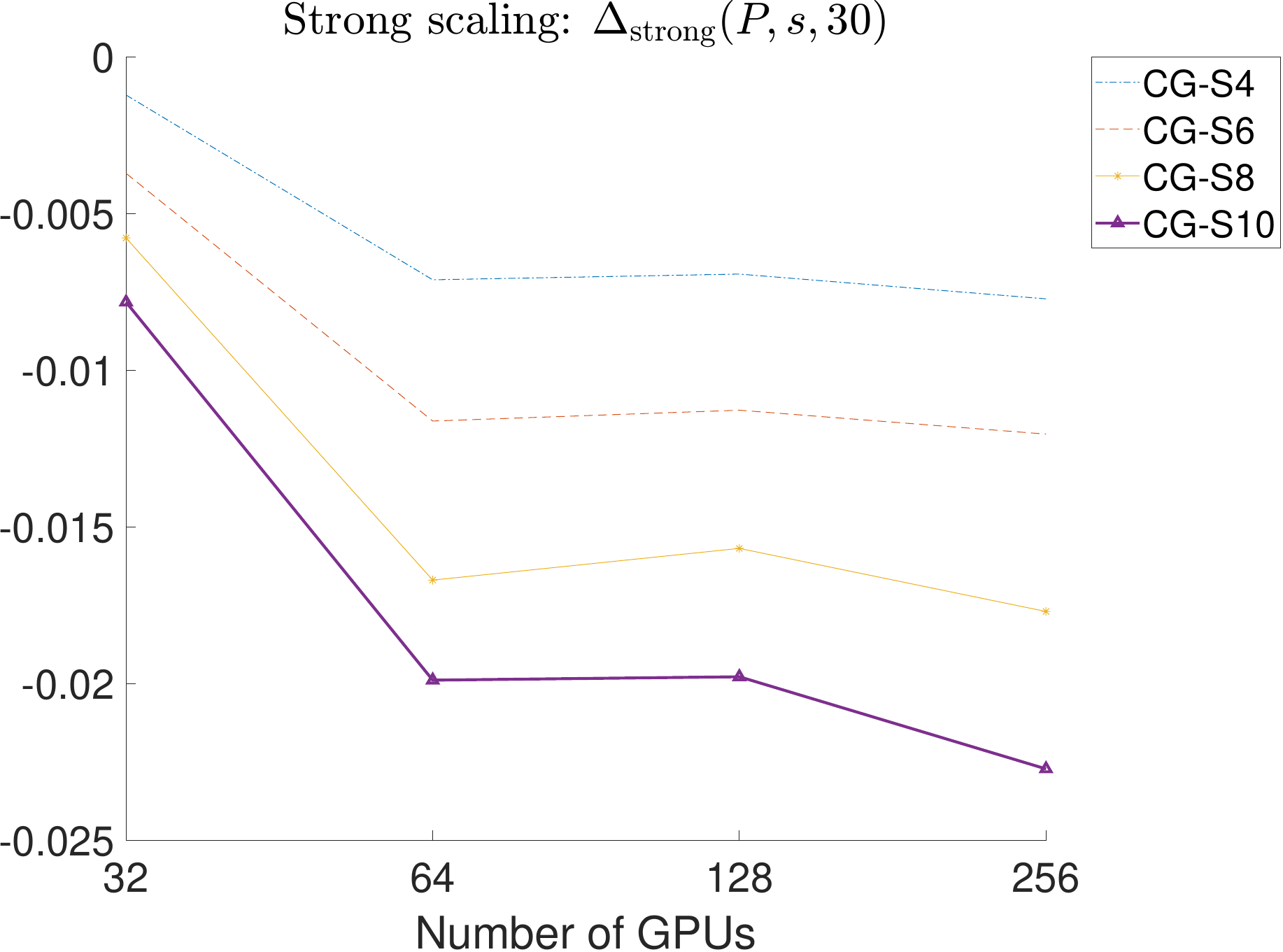}
\caption{Strong--scaling behavior of
  \(\Delta_{\mathrm{strong}}(P,s,\nu=30)\) with fixed \(n=500^3\) \label{fig:strong1}}
\end{figure}
\begin{figure}[!htbp]
\centering
\includegraphics[width=0.7\textwidth]{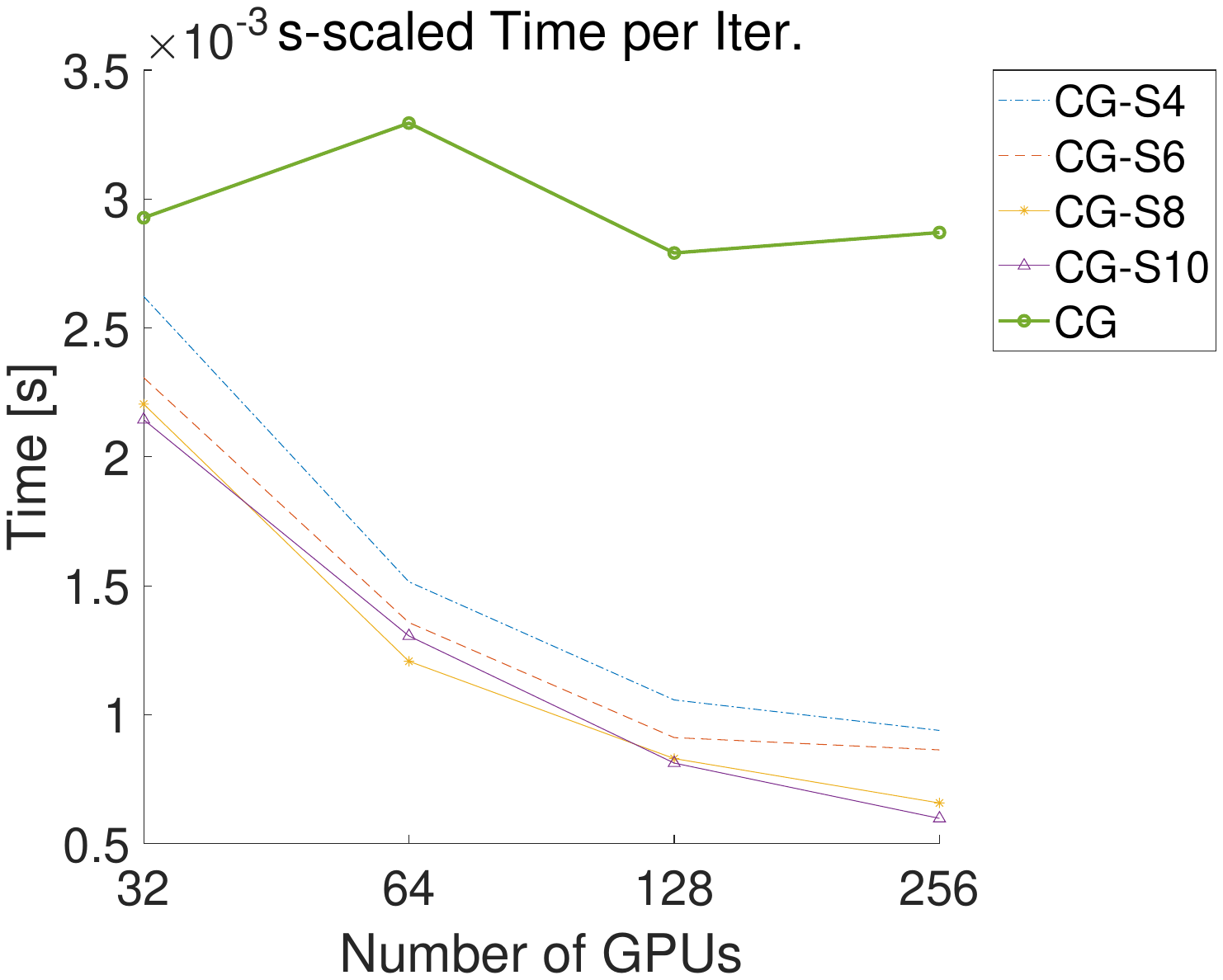}\\[0.2cm]
\includegraphics[width=0.7\textwidth]{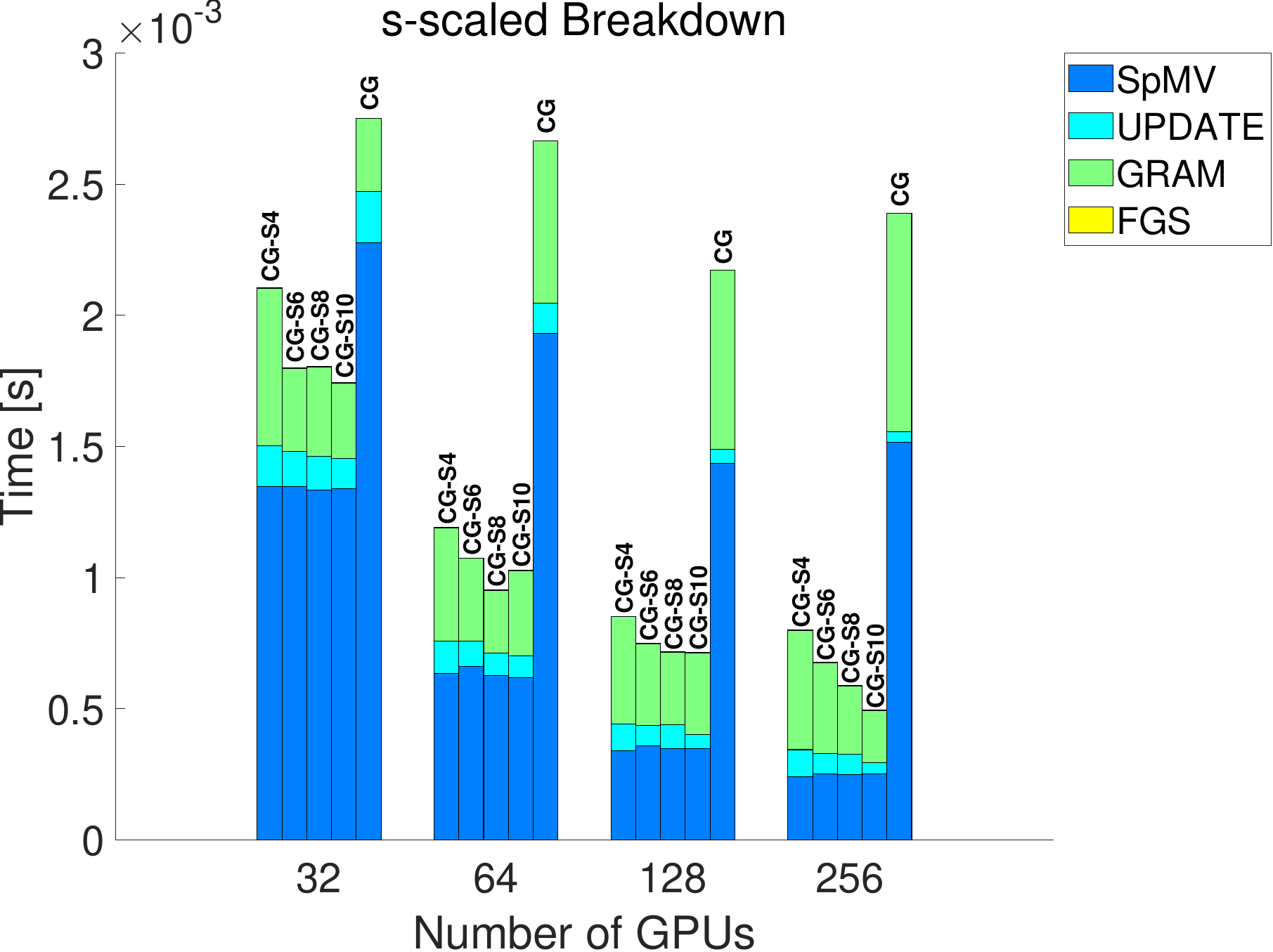}
\caption{Strong scalability: time per iteration scaled by step size s (top) and related breakdown (bottom). \label{fig:strong2}}
\end{figure}
Figure~\ref{fig:strong1} reports the strong--scaling indicator
$\Delta_{\mathrm{strong}}$ measured on MareNostrum~5 for the CG-S variants, with different values of $s$, as a function of the number of GPUs. For $P=32$, the cost model predicts positive values of $\Delta_{\mathrm{strong}}$ for $s \ge 8$, suggesting that the additional computational work of the $s$-step formulation should dominate at this scale. In contrast, the measured data show that $\Delta_{\mathrm{strong}}$ is already negative in this regime. Moreover, the magnitude of the measured negative difference is
approximately two orders of magnitude larger than predicted by the
model. Thus, the simplified model substantially underestimates the
performance advantage of PCG--S, possibly because it does not fully
capture the communication overhead or the different effective
execution rates of the local kernels. The measurements nevertheless
show that communication avoidance provides tangible benefits, even
at moderate scale, as $s$ increases.

In agreement with the model, as $P$ grows the magnitude of $\Delta_{\mathrm{strong}}$ increases with $s$, and the difference becomes more pronounced at larger scales. This trend highlights the increasing impact of global reductions as both $P$ and $s$ grow, confirming that communication-avoiding strategies are particularly effective in the large--scale regime.

Fixing the iteration count isolates system--level effects and shows that reducing the number of global reductions per iteration directly improves strong--scaling efficiency. When normalized by $s$, the per--iteration time of the $s$--step CG method decreases nearly monotonically with $P$ for larger $s$, indicating that the additional arithmetic work is effectively amortized by the reduction in communication latency.

The breakdown analysis in Figure~\ref{fig:strong2} further shows that, as $P$ increases, global reductions associated with dot products account for an increasingly large fraction of the runtime in classical CG. The $s$--step formulation mitigates this cost substantially. SpMV remains the dominant kernel, while vector operations and reductions are significantly reduced in the $s$-step setting. The FGS Gram solve is negligible in these tests, representing under $1\%$ of the per--iteration solve time even for the largest $s$ (Figure~\ref{fig:strong2}, bottom panel).

For completeness, Figure~\ref{fig:strong3} reports the speedup of CG and the CG-S variants. While CG exhibits limited strong scalability, the CG-S variants scale increasingly better as $s$ grows, consistently with the reduced frequency of global communications. 

\begin{figure}[!htbp]
\centering
\includegraphics[width=0.7\textwidth]{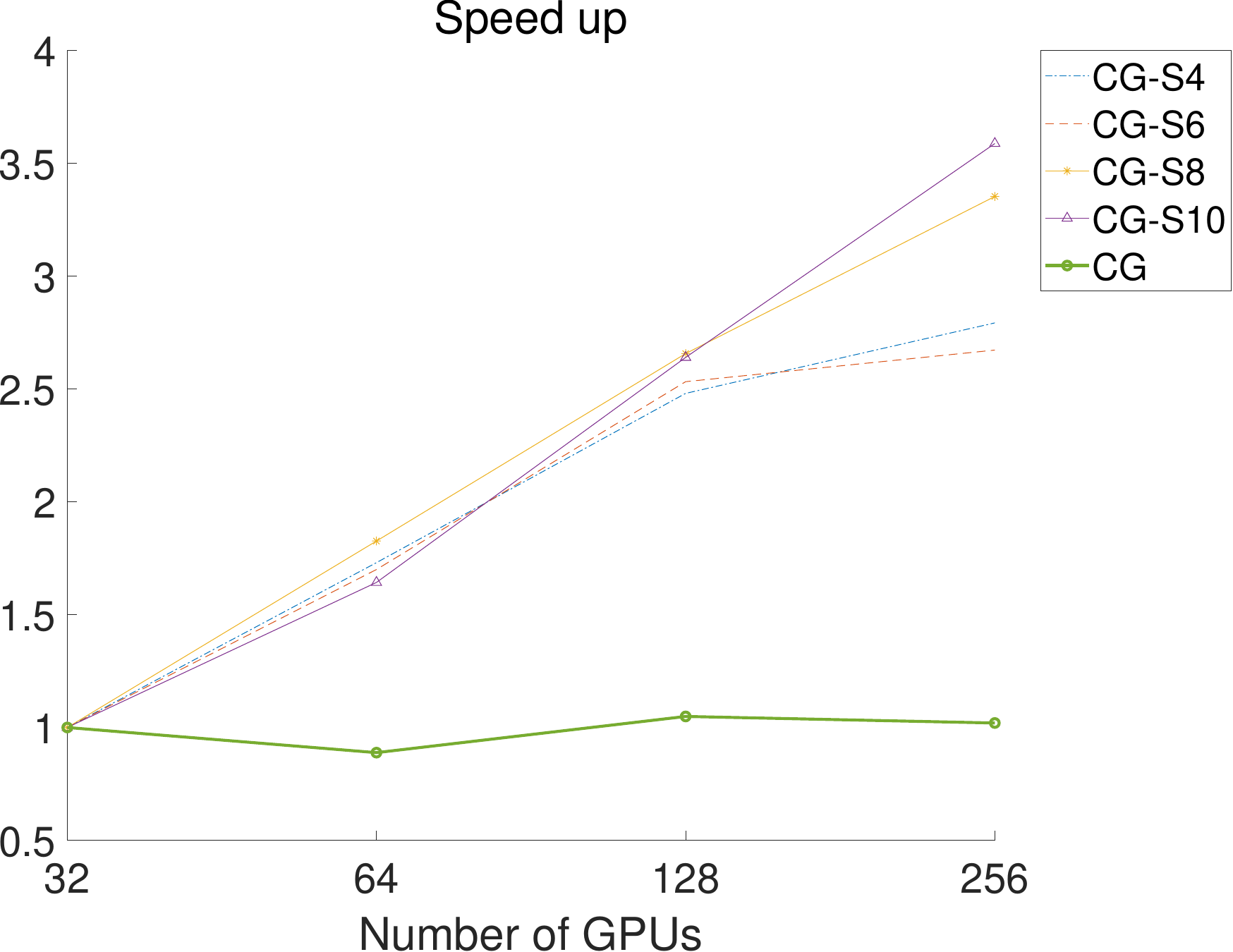}
\caption{Strong scalability: Speed up. \label{fig:strong3}}
\end{figure}

\subsection{Weak Scalability}
\label{weak}

In this section, we evaluate the weak scalability of the solver by increasing the number of DOFs proportionally to the number of GPUs, while keeping the local problem size fixed at $200^3$ DOFs per GPU. Experiments were performed on up to 512 GPUs of the Leonardo supercomputer, solving linear systems with more than $4$ billion DOFs in total.

To assess algorithmic scalability and numerical robustness, we consider the preconditioned variant of PCG-S, employing the Algebraic MultiGrid (AMG) method implemented in \texttt{BootCMatchGX} and described in~\cite{ieee2023}, configured with a W-cycle. As in the parametric study and previous experiments, the reduced Gram systems are solved using $\nu = 30$ FGS sweeps.

For all test cases, the solver is run with relative residual
tolerance $\tau=10^{-6}$ and maximum outer-iteration count
$k_{\max}=1000$. Results are compared against the corresponding classical PCG method.

\begin{figure}[!htbp]
\centering
\includegraphics[width=0.7\textwidth]{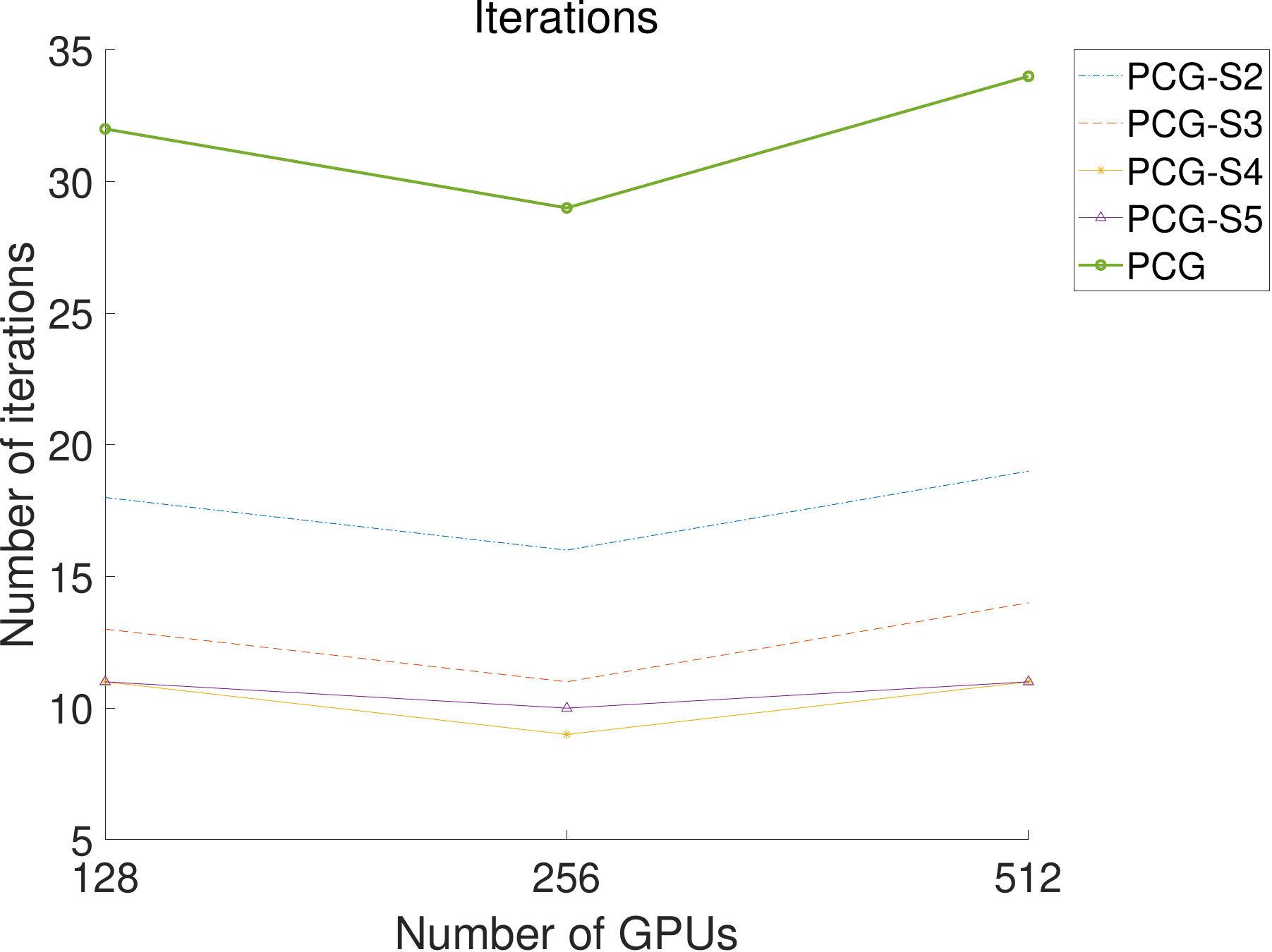}\\[0.2cm]
\includegraphics[width=0.7\textwidth]{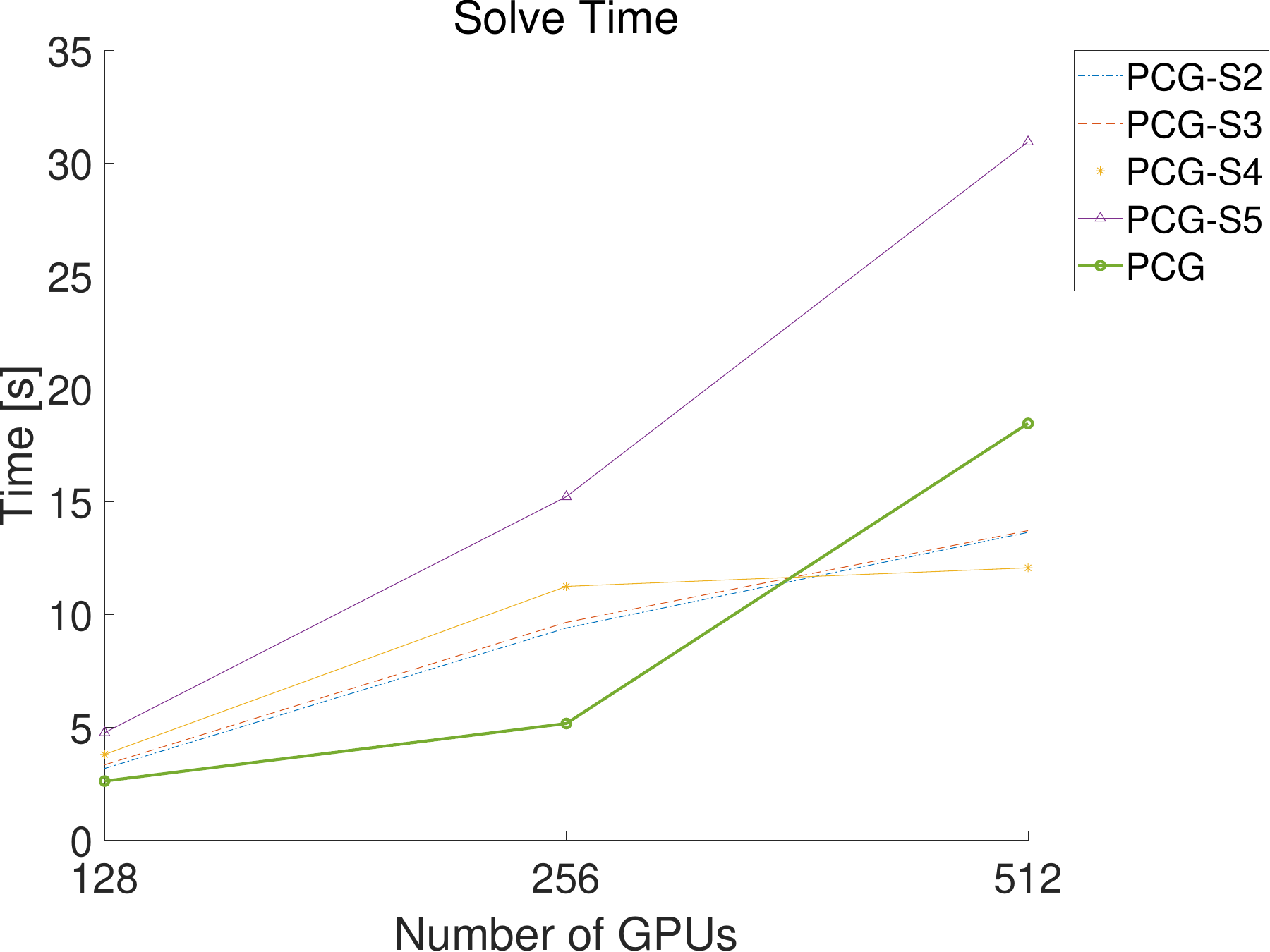}
\caption{Weak scalability: number of iterations (top) and total solve time (bottom).}
\label{fig:weak1}
\end{figure}

Figure~\ref{fig:weak1} reports the weak--scaling behavior in terms of outer iterations (top) and total solve time (bottom). Each PCG--S outer iteration represents approximately $s$ Krylov
steps, involving $s$ operator applications and $s$ preconditioner
applications. Consequently, the smaller PCG--S outer-iteration counts
and their decrease as $s$ increases reflect the expected aggregation
of multiple Krylov steps within each outer iteration. The plot is
intended to document the convergence behavior of each configuration
under weak scaling, rather than to compare equivalent Krylov work
across different values of $s$.

For each fixed value of $s$, the outer iteration count remains
relatively stable as the number of GPUs and the global problem size
increase. Since the amount of Krylov work represented by each outer
iteration remains unchanged when $s$ is fixed, this behavior provides
evidence of algorithmic weak scalability. It also indicates that the
Chebyshev-based PCG--S formulation combined with the AMG
preconditioner preserves robust convergence over the tested range of
problem sizes and GPU counts.

The main performance result concerns time-to-solution. Indeed,
at $512$ GPUs, the reduced frequency of global synchronization for moderate step sizes $s=2,3,4$ outweighs 
the additional local block work and any modest difference in operator applications, producing a measured reduction in time-to-solution.

\begin{figure}[!htbp]
\centering
\includegraphics[width=0.7\textwidth]{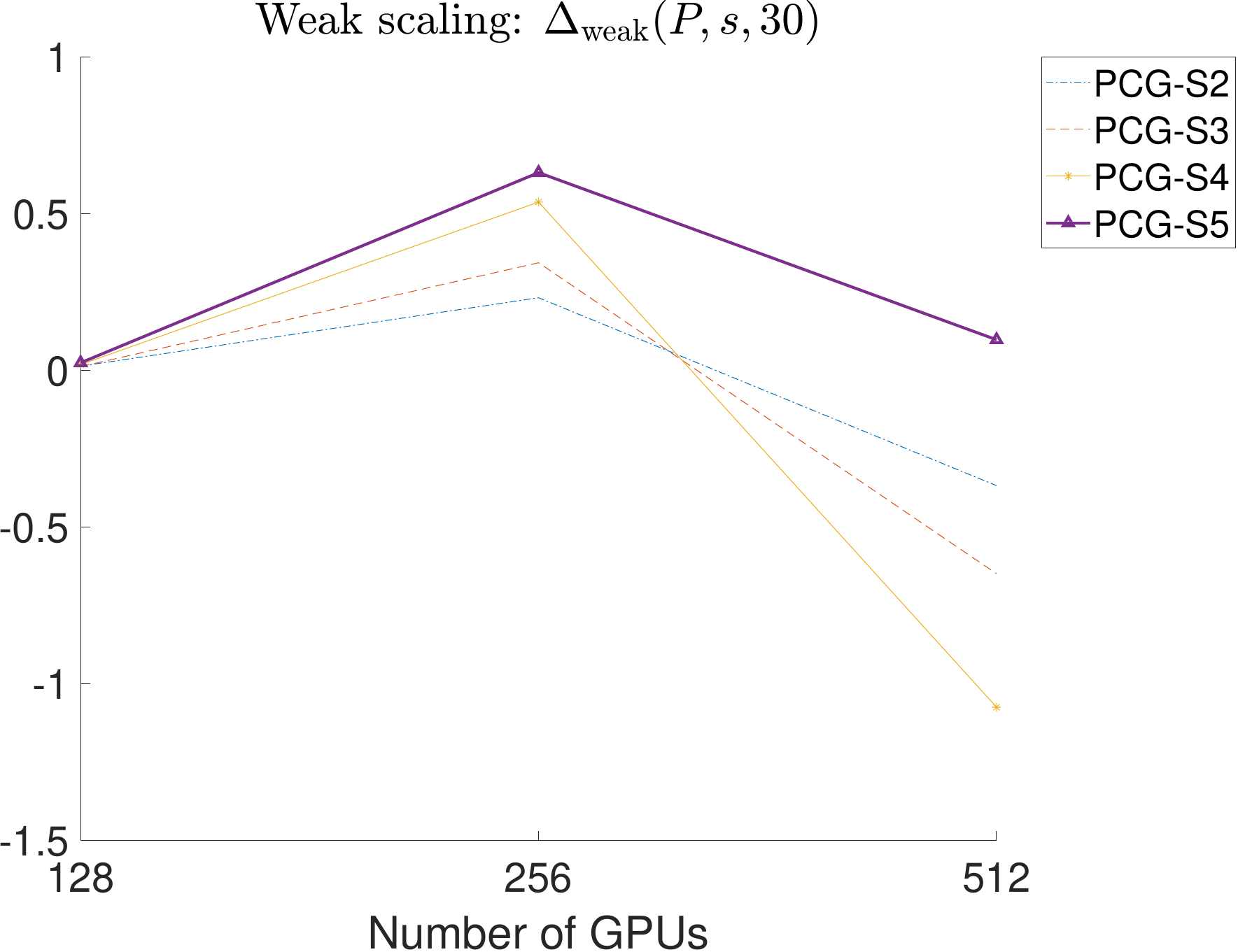}\\[0.2cm]
\includegraphics[width=0.7\textwidth]{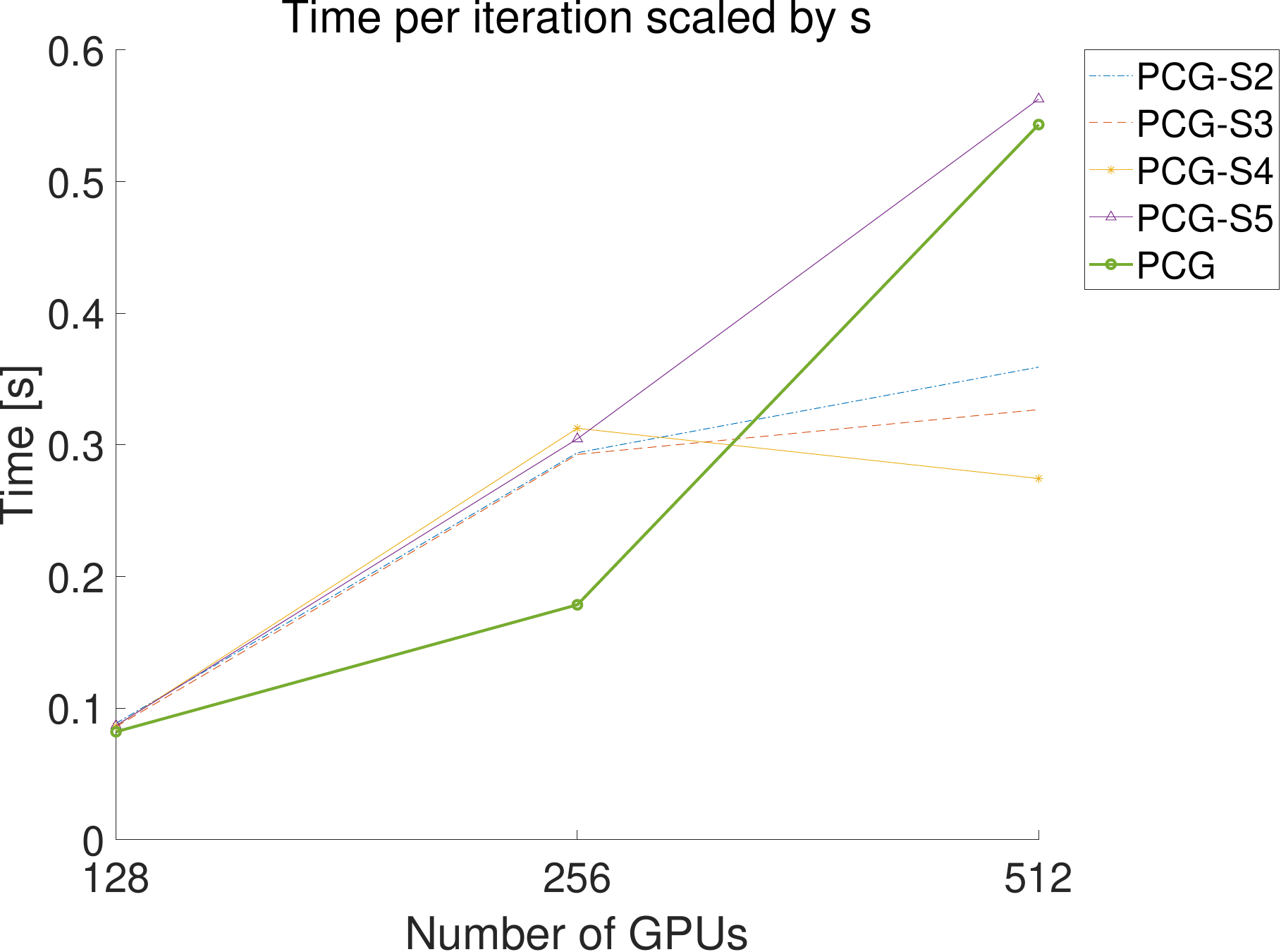}
\caption{Weak-scaling behavior of $\Delta_{\mathrm{weak}}(P,s,\nu=30)$ (top) and time per iteration normalized by the step size $s$ (bottom).}
\label{fig:weak3}
\end{figure}

Figure~\ref{fig:weak3} (top) reports the weak-scaling indicator
$\Delta_{\mathrm{weak}}(P,s,\nu=30)$. As the number of GPUs increases,
the indicator decreases and becomes negative at $512$ GPUs for
$s = 2, 3,$ and $4$. A negative value indicates that one PCG--S outer iteration is faster
than $s$ classical PCG iterations under the model assumptions.
This behavior is consistent with the performance model of
Section~\ref{sec:costperiter}, which predicts that communication savings
become dominant at sufficiently large $P$.

The discrepancy observed for $s=5$ at 512 GPUs is plausibly attributable
to increased sensitivity to system noise and tail--latency effects in
global collectives at extreme scale. A more detailed analysis at larger
process counts would be required to fully characterize this behavior.

The bottom panel of Figure~\ref{fig:weak3} reports the scaled time per outer
iteration, showing a trend consistent with the weak--scaling indicator
in the top panel. For $s = 2, 3,$ and $4$, the scaled per--iteration time
decreases at $512$ GPUs, confirming that communication reduction effectively
compensates for the additional local computation introduced by the
$s$-step formulation. Combined with the reduced number of outer iterations
observed in Figure~\ref{fig:weak1}, this explains the improved time-to-solution
achieved by these configurations.

\begin{figure}[!htbp]
\centering
\includegraphics[width=0.7\textwidth]{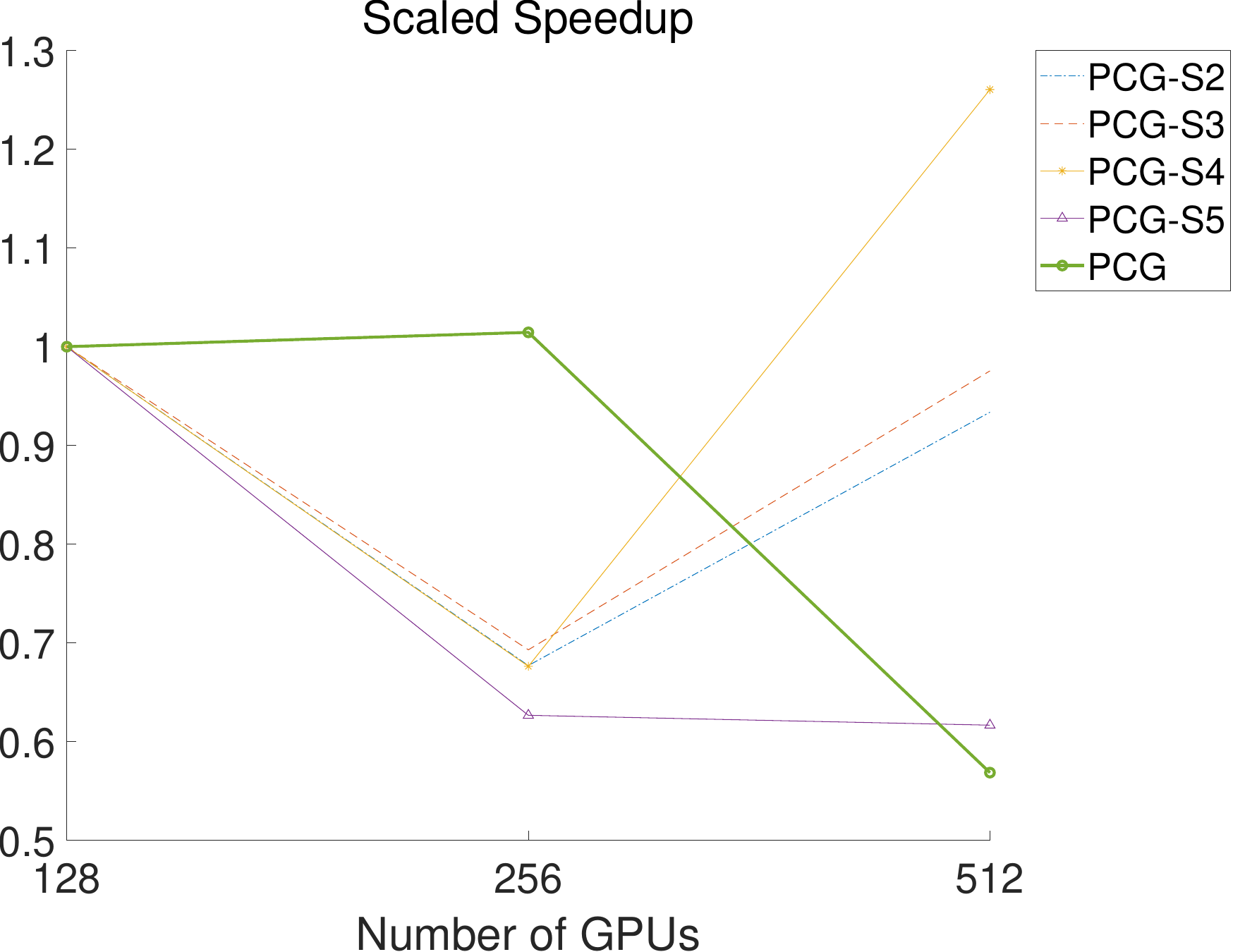}
\caption{Weak scalability: scaled speedup.}
\label{fig:weak4}
\end{figure}

Figure~\ref{fig:weak4} reports the scaled speedup under weak scaling.
At 512 GPUs, PCG exhibits a noticeable performance degradation,
indicating that synchronization costs and collective latency effects
increasingly limit scalability at extreme concurrency.

In contrast, the $s$-step variants with $s = 2, 3,$ and $4$ show an
increasing scaled speedup as the number of GPUs grows. Among these,
$s = 4$ achieves the best overall performance, providing the most
favorable balance between communication reduction and additional
local computation.

These results indicate that moderate step sizes offer the best
trade-off at 512 GPUs. Verifying whether the theoretical model, which
predicts increasing benefits for larger block sizes as $P$ grows, continues to hold beyond the tested regime will require experiments
at larger scales.

Overall, no severe conditioning-related instability was observed for the
Chebyshev-based blocks in the tested configurations, while the FGS
solution of the Gram system remains numerically stable with a very negligible
overhead in the tested regime, validating the design choice of combining these components.

Under weak scaling, the $s$-step formulation
improves time-to-solution at 512 GPUs for moderate step sizes
($s = 2, 3, 4$), where the reduction in outer iterations compensates
for the additional local computation. In this regime, $s = 4$ provides
the most favorable balance between communication reduction and
arithmetic overhead.

These findings suggest the existence of an optimal moderate step
size in the explored range. At larger GPU counts, the interplay
between communication and computation will remain the determining
factor, and careful selection of $s$ will be essential to fully
exploit the advantages of the $s$-step approach.

\section{Concluding Remarks and Future Work}
\label{sec:conc}

We presented a scalable multi--GPU implementation of the $s$--step PCG method combining a Chebyshev--stabilized Krylov basis with an iterative FGS solver for the reduced Gram systems.

The experimental results indicate that the method exhibits satisfactory numerical robustness within the tested configurations, and that the proposed performance model reflects the main observed scalability trends. In strong scaling, the $s$--step variants tend to become competitive with, and in several cases faster than, classical CG as the number of GPUs increases, suggesting that synchronization costs play an increasingly relevant role at scale. In weak scaling with AMG preconditioning, reductions in time--to--solution are observed at 512 GPUs for moderate step sizes ($s=2,3,4$), with $s=4$ offering a favorable compromise between communication reduction and additional arithmetic cost in the explored regime. Convergence has been observed on problems exceeding $4$ billion degrees of freedom.

For moderate values of $s$, the Chebyshev--based blocks appear to remain sufficiently well conditioned, and a fixed moderate number of FGS sweeps provides an adequate and inexpensive approximation of the Gram solves in the tested scenarios, without evidence of a significant performance penalty. For larger--scale settings, an adaptive Gram solver strategy based on the relative residual of the reduced system, allowing the inner accuracy to vary across outer iterations in accordance with inexact Krylov principles, could offer additional flexibility in balancing robustness and efficiency.

Further investigation at larger node counts and for larger step sizes is required to better characterize the method in more extreme regimes. Within the scales currently examined, the implementation appears to make effective use of modern GPU--accelerated architectures. Since communication is typically more energy--intensive than computation, reducing global synchronization and increasing computational locality may also have favorable implications for energy efficiency, although a dedicated study would be necessary to assess this aspect quantitatively.

Since robust preconditioning generally represents a dominant portion of the runtime, future work should also address the design of robust AMG--like strategies that preserve spectral regularization while minimizing synchronization overhead. Reducing communication within the preconditioner, increasing local computational density, and exploring adaptive precision mechanisms are key directions to further improve scalability and time--to--solution on GPU platforms.

Beyond the algorithmic and theoretical contributions, the work provides,
to the best of our knowledge, the first publicly available fully distributed
multi--GPU implementation of preconditioned $s$-step CG at the scale
reported here.

\section*{Conflict of Interest}
The authors declared that they have no conflict of interest.

\section*{Funding statement}
This work was partially supported by Spoke~6 “Multiscale Modelling \& Engineering Applications” 
of the Italian Research Center on High-Performance Computing, Big Data and Quantum Computing (ICSC), 
funded by MUR under the NextGenerationEU (NGEU) program, and by the EuroHPC project 
“Energy Oriented Center of Excellence (EoCoE III): Fostering the European Energy Transition with Exascale” 
(Project No.~101144014), funded by the European Commission.

\section*{Author Contributions}
Pasqua D'Ambra and Stephen Thomas conceptualized the work and wrote the manuscript. Pasqua D'Ambra defined the algorithmic framework and analyzed the results. Massimo Bernaschi and Mauro G. Carrozzo implemented the software, and Mauro G. Carrozzo carried out the experimental testing.

\section*{Acknowledgments}
The authors acknowledge Spoke~6 of ICSC for granting access to the Leonardo system at CINECA (Italy), and the EuroHPC JU for granting access to MareNostrum~5 at BSC (Spain) under the EuroHPC Development Access call.
These resources supported the computational experiments presented in this work.

\bibliography{CGS-step}

\end{document}